\documentclass[12pt,reqno]{amsart}

\usepackage[utf8]{inputenc}
\usepackage{amssymb}
\usepackage{amsmath}
\usepackage{amsthm}
\usepackage{amsfonts}
\usepackage{bbm}
\usepackage{bookmark}
\usepackage{hyperref}
\hypersetup{pdfstartview={FitH}}
\usepackage{color}
\usepackage{xcolor}
\usepackage{mathrsfs} 
\usepackage[normalem]{ulem}
\usepackage{tikz}
\usepackage{tikz-3dplot}
\usetikzlibrary{trees,shapes,backgrounds}
\usepackage{caption}
\usetikzlibrary{arrows, arrows.meta, calc, positioning}
\usetikzlibrary{shapes.multipart}

\theoremstyle{plain}
\newtheorem{theorem}{Theorem}[section]
\newtheorem{lemma}[theorem]{Lemma}
\newtheorem{prop}[theorem]{Proposition}
\newtheorem{corollary}[theorem]{Corollary}
\newtheorem{proposition}[theorem]{Proposition}

\theoremstyle{definition}

\numberwithin{equation}{section}
\newtheorem*{theorem*}{Theorem} 

\newcommand{\Z}{{\mathbb Z}}
\newcommand{\R}{{\mathbb R}}

\def\wh{\widehat}

\DeclareFontFamily{U}{mathx}{\hyphenchar\font45}
\DeclareFontShape{U}{mathx}{m}{n}{
<5> <6> <7> <8> <9> <10>
<10.95> <12> <14.4> <17.28> <20.74> <24.88>
mathx10
}{}
\DeclareSymbolFont{mathx}{U}{mathx}{m}{n}
\DeclareFontSubstitution{U}{mathx}{m}{n}
\DeclareMathAccent{\widecheck}{0}{mathx}{"71}

\title[Norm-variation of triple ergodic averages]{Norm-variation of triple ergodic averages for commuting transformations}

\author[P. Durcik]{Polona Durcik}
\address{Schmid College of Science and Technology, \newline  Chapman University, One University Drive, Orange, CA 92866, USA}
\email{durcik@chapman.edu}

\author[L. Slav\'ikov\'a]{Lenka Slav\'ikov\'a}
\address{Department of Mathematical Analysis, Faculty of Mathematics and Physics, \newline Charles University, Sokolovsk\'a 83, 186 75 Praha 8, Czech Republic}
\email{slavikova@karlin.mff.cuni.cz}

\author[C. Thiele]{Christoph Thiele}
\address{Mathematisches Institut,\newline Universit\"at Bonn, Endenicher Allee 60,  53115 Bonn, Germany}
 \email{thiele@math.uni-bonn.de}

\date{\today}

\begin{document}
\tdplotsetmaincoords{70}{110}

\begin{abstract}
    We prove an $r$-variation estimate, $r>4$, in the norm for ergodic averages with respect to three commuting transformations. It is not known whether
    such estimates hold for all $r\ge 2$ as in the analogous cases for one or two commuting transformations, or whether such estimates hold for any $r<\infty$ for more than three commuting transformations.
\end{abstract} 
 
\maketitle

\section{Introduction}

\label{introsection}

We prove the following norm variation bound for three commuting transformations.
\begin{theorem}\label{thm:ergodicthm}
For all $r>4$, there exists a constant $C>0$ such that the following holds. Let $(X,\mathcal{F},\mu)$ be a $\sigma$-finite measure space,  $T_0,T_1,T_2\colon~X\to X$   mutually commuting measure preserving transformations and let $J$ and $n_0<n_1<\cdots<n_J$ be positive integers. For any   $f_0,f_1\in L^8(X)$ and $f_2 \in L^4(X)$,  each of respective norms one,  we have the bound
\begin{equation*}
\sum_{j=1}^{J} \|M_{n_j}(f_0,f_1,f_2) - M_{n_{j-1}}(f_0,f_1,f_2)\|_{{L}^2(X)}^r \leq C ,
\end{equation*}
where we have defined for almost every $x\in X$ 
\begin{equation*}
M_{n}(f_0,f_1,f_2)(x) := \frac{1}{n}\sum_{i=0}^{n-1} f_0(T_0^i x) f_1(T_1^i x) f_2(T_2^i x).
\end{equation*}
\end{theorem}

Norm variation bounds with $r\ge 2$ for one transformation were proven in \cite{MR1389625}
and for two commuting transformations in \cite{MR3904183},
following earlier work \cite{MR3480347} in the finite group setting. 
Norm variation bounds with any $r<\infty$  for any number of commuting transformations were stated as an open problem in the closing section of \cite{MR3345161}. 
Any such bounds remain unknown for more than three
commuting transformations. 
It is natural to conjecture norm variation bounds for
$r\ge 2$ for any number of commuting transformations. The passage from two to three commuting transformations is a critical transition as present techniques very clearly fail to address the sharp variation threshold $r\ge 2$.

Norm variation bounds for any $r$ are strong quantitative forms of norm convergence. 
Qualitative norm convergence for 
three or more commuting transformations was proven by Tao in \cite{MR2408398} by finitary methods.
The case for two commuting transformations had been shown before using the tools from ergodic theory. Ergodic theoretic proofs of Tao's result were given in \cite{MR2599882}, \cite{MR2539560}, and a generalization 
to transformations generating a nilpotent group was proven in \cite{MR2912715}.

Norm convergence should be compared with the more difficult question
of pointwise convergence almost everywhere. Such pointwise
convergence is known by the classical Birkhoff theorem for
a single transformation \cite{bir30}, with pointwise variational estimates proven in \cite{MR1019960}.
Pointwise convergence almost everywhere remains a widely recognized open problem even in the case of two general commuting transformations.
This contrasts with recent developments  in the area concerning multiple ergodic averages with actions 
of polynomial powers $T^{p(n)}$, including a number
of pointwise almost everywhere convergence results 
under the umbrella of the Furstenberg-Bergelson-Leibman conjecture such as the bilinear but
not completely linear polynomial averages in \cite{MR4413747} or
the multi-parameter polynomial averages in \cite{BMSW22}. 
For further history on the ergodic means discussed in the present paper, we refer to the paper on two commuting transformations~\cite{MR3904183}.

By a variant of the well known Calder{\'o}n transference principle,  
 Theorem \ref{thm:ergodicthm} follows from 
Theorem \ref{thm:2} below. We do not elaborate on the transference principle in the present paper but refer to
the case of two commuting transformations  in \cite{MR3904183}. It reduces quantitative convergence results to
analogous results on individual orbits of the action of the group  spanned by the commuting transformations
and parameterized by $\mathbb{Z}^3$.
The further transfer from
$\mathbb{Z}^3$ to $\mathbb{R}^3$ as in Theorem
\ref{thm:2} is harmless and it can be made a part of the transference principle in our setting, unlike in the setting of actions $T^{p(n)}$
with polynomials of higher degree which face number theoretic complications.

\begin{theorem}\label{thm:2}
For all $r>4$, there exists a  constant $C>0$ such that the following holds. For any positive integer $J$ and positive real numbers $t_0<t_1<\cdots < t_J$,  any $f_0,f_1\in L^8(\R^3)$  and $f_2\in {L}^4(\mathbb{R}^3)$ with respective norms one, we have 
\begin{equation}\label{e:main2} \sum_{j=1}^{J} \|M_{t_j}(f_0,f_1,f_2) - M_{t_{j-1}}(f_0,f_1,f_2)\|_{{L}^2(\R^3)}^r \leq C 
\end{equation}
where, with $e_0,e_1,e_2$ the standard unit vectors in $\R^3$, we have defined for almost every $x\in \R^3$:
\begin{equation}\label{eq:averages}
M_{t}(f_0,f_1,f_2)(x) := \frac{1}{t}\int_0^t f_0(x+\tau e_0) f_1(x+\tau e_1) f_2(x+\tau e_2)\,d\tau.
\end{equation}
\end{theorem}

Only the choice of tuple of exponents $(8,8,4)$ breaks the symmetry between the three functions in the above theorems. One therefore concludes the analogous estimates for permutations of these exponents. Interpolation gives further
tuples of exponents, for example the symmetric tuple $(6,6,6)$.

Theorem \ref{thm:2} is proven 
using
the theory of singular Brascamp-Lieb forms. 
A singular Brascamp-Lieb datum $D=(n,S,\Pi,(\Pi_s)_{s\in S})$ is a tuple containing the dimension $n\ge 1$ of the domain of integration, the finite set $S$ parameterizing the tuple of input functions,
and linear maps $\Pi$ and $\Pi_s$ for $s\in S$ on the domain $\R^{n}$,
where $\Pi_s$ maps onto the domain of the input function with parameter $s$, typically of smaller dimension than $n$. Together with some singular integral kernel $K$
on the range of $\Pi$, the singular Brascamp-Lieb 
form $\Lambda_{D,K}$ is defined as
\[\Lambda_{D,K}((f_s)_{s\in S})=\int_{\R^n} K(\Pi x)\prod_{s\in S} f_s(\Pi_sx)\, dx ,\]
where the integral is defined in some principal value sense or, if the kernel has additional qualitative regularity as is mostly the case in the present paper, in the Lebesgue integral sense. We also often talk about the multiplier $m$ of the form,
which is the Fourier transform of the kernel $K$. A singular Brascamp-Lieb inequality estimates this form by a constant times
the product of Lebesgue norms $\prod_{s\in S} \|f_s\|_{p_s}$ for some tuple of
exponents $p_s$.

Singular Brascamp-Lieb inequalities with the kind of data appearing in this paper are
studied in \cite{MR4041095}, \cite{MR4510172},  \cite{MR4390229}, and \cite{MR4478297}  when $K$ is a classical Calder\'on-Zygmund kernel. Compared with this work, the novelty in the present paper is that the kernels $K$ do not satisfy uniform Calder\'on-Zygmund bounds but rather multi-parameter symbol estimates arising naturally in the investigation of variation norms.
These symbol estimates no longer synchronize with a Whitney decomposition of frequency space but rather involve regions determined by
an arbitrary sequence of jumps between scales, 
  such as the red regions with arbitrary eccentricity in Figure
\ref{fig:1}.
\begin{figure}[htb]
\centering
\includegraphics{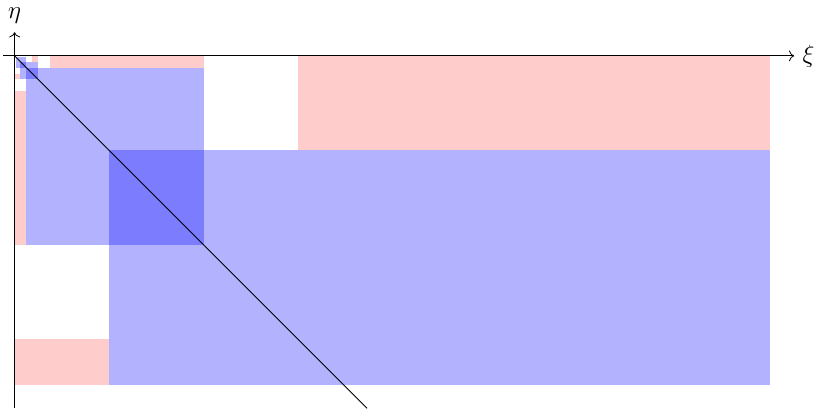}
\caption{Structure of $m=\widehat{K}$ in  Theorem~\ref{thm:2}.} \label{fig:1}
\end{figure}

Multi-parameter singular Brascamp-Lieb forms of this type appear in more basic form already in the case of two commuting transformations  \cite{MR3904183}. 
Compared to two transformations, novel challenges for three commuting transformations arise from the absence of the cubical structure of the main singular Brascamp-Lieb form relevant to  Theorem~\ref{thm:2}. For two commuting transformations the set $S$ of the Brascamp-Lieb datum can be naturally identified with the corners of a square, but for three commuting transformations it can not be identified with corners of a cube, but rather with vertices of a triangular prism.  Cubical structure is important to allow a loss-free symmetrization of the form along the reflection symmetries of the cube. Lacking
such cubical structure,  the techniques available to us lead to an unavoidable loss analogous to the work on cancellation for the simplex Hilbert transform \cite{MR4041095}
and also \cite{MR4409885}, \cite{DS22}.
One novelty in the present paper is that this loss needs to be absorbed by a relaxation of the variation norm parameter $r$ towards $r>4$.
In other words, we cannot allow a loss in difference between largest and smallest scale involved, i.e., in the total
number of intermediate scales involved, but only a loss in the much smaller number of jumps between the scales. Thus we need to develop an analysis that carefully uses and preserves the
particular structure of multipliers as depicted in Figure 
\ref{fig:1} throughout the argument.

We next provide an overview over the arguments of the present paper,
which is structured into intermediate propositions and sections as in Figure
\ref{fig:tree}.

\tikzstyle{line} = [draw, thick,-latex']

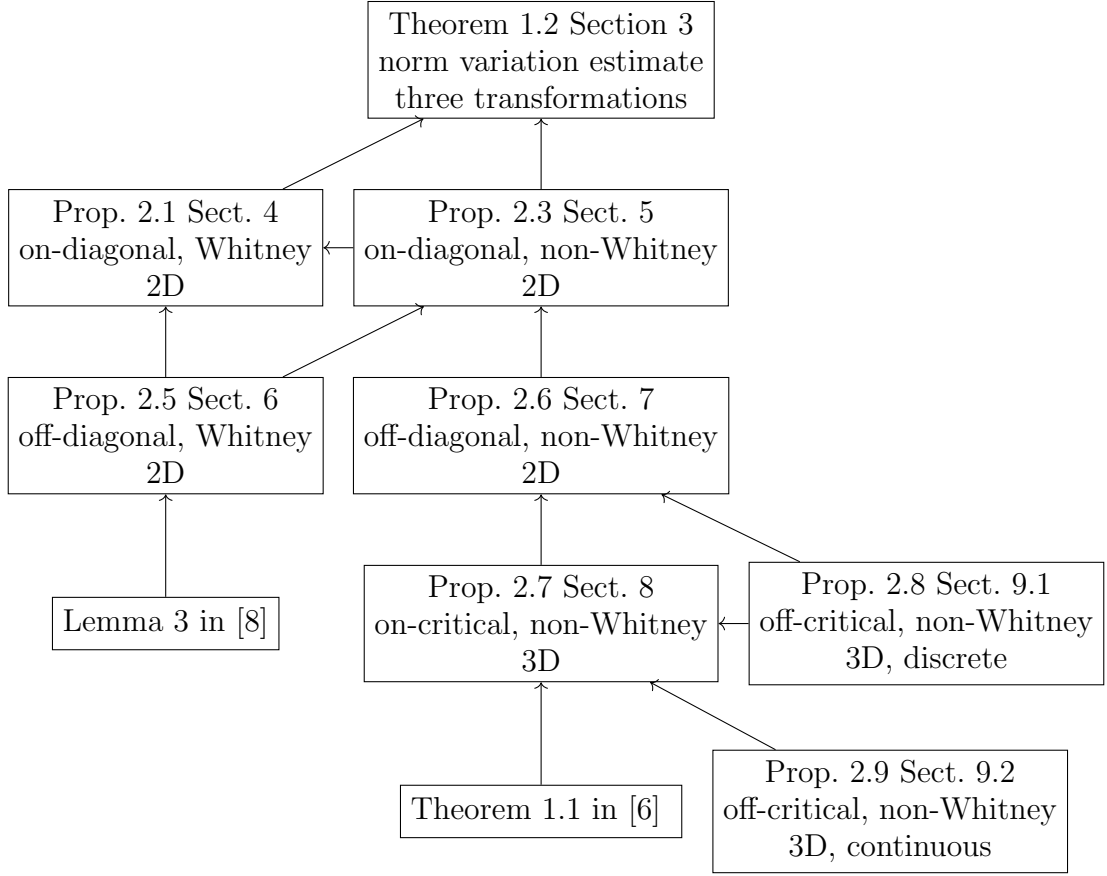
\begin{figure}[htb]
    \centering
\begin{tikzpicture}[scale=1, node distance={2.5cm},   every node/.style={draw, rectangle,  align=center},every text node part/.style={align=center}]

\node (12) {Theorem \ref{thm:2} Section \ref{proofthm:2}\\ norm variation estimate 
\\
three transformations};
\node  (23) [below of = 12] {Prop.~\ref{prop:2} Sect.~\ref{sec:prop:2}\\ on-diagonal, non-Whitney\\ 2D};
\node  (21) [left  = 0.4cm of 23] {Prop.~\ref{prop:1} Sect.~\ref{sec:prop:1}\\on-diagonal, Whitney\\ 2D};
\node  (25) [below    of= 21] {Prop.~\ref{firstcz} Sect.~\ref{sec:firstcz}\\ off-diagonal, Whitney\\ 2D};
\node  (26) [below   of = 23] {Prop.~\ref{firststick} Sect.~\ref{sec:firststick}\\ off-diagonal, non-Whitney\\ 2D};
\node  (L3) [below   of = 25] {Lemma 3  in~\cite{MR4171366}};
\node  (27) [below  of = 26] {Prop.~\ref{thm:4} Sect.~\ref{sec:thm:4}\\ on-critical, non-Whitney\\ 3D};
\node  (T11) [below  of = 27] {Theorem 1.1 in~\cite{MR4510172} };
\node  (28) [right = 0.4cm of  27] {Prop.~\ref{secstickphi} Sect.~\ref{sec:9.1}\\  off-critical, non-Whitney \\  3D, discrete};
\node  (29) [right = 0.4cm of T11] {Prop.~\ref{secstickgauss} Sect.~\ref{sec:9.2}\\   off-critical, non-Whitney \\  3D, continuous};

\draw[->] (21) -- (12);
\draw[->] (23) -- (12);
\draw[->] (23) -- (21);
\draw[->] (25) -- (23);
\draw[->] (26) -- (23);
\draw[->] (L3) -- (25);
\draw[->] (T11) -- (27);
\draw[->] (27) -- (26);
\draw[->] (28) -- (26);
\draw[->] (28) -- (27);
\draw[->] (29) -- (27);
\draw[->] (25) -- (21);
\end{tikzpicture}
\caption{Structure of the proof of the main theorem}
\label{fig:tree}
\end{figure}

Theorem~\ref{thm:2} 
is deduced  in Section~\ref{proofthm:2} from~\eqref{eq:doublenormvar}, an  estimate in terms of a fixed number $J$ of jumps in the variation,  which can be thought of as an endpoint estimate
at $r=4$ for~\eqref{e:main2}. This endpoint estimate is reduced to two singular Brascamp-Lieb estimates, both  with datum $D_1$ defined in \eqref{defd1}, but with different two-dimensional kernels illustrated in Figure \ref{fig:1}.
For simplicity we focus on one quadrant in our
discussion, as the other quadrants do not pose additional difficulties.
The first singular Brascamp-Lieb estimate, Proposition \ref{prop:1}, takes care of the so-called short variation with a multiplier that lives near the dark
blue overlap regions of the light blue squares in Figure \ref{fig:1}.
The size of each dark blue square  is comparable to its distance to the origin, a property we call 
Whitney. Proposition \ref{prop:2} takes care of the so-called long variation with multipliers living
at the light blue squares themselves. Each light blue square includes potentially many scales and therefore is not in general Whitney. However,
the piece of the multiplier associated with each light blue square has elementary tensor structure and telescopes into 
the difference between its largest and its smallest scale.
For both of these propositions, it is important that the number of squares involved 
is controlled by $J$.

\begin{figure}[htb]
    \centering
 \includegraphics{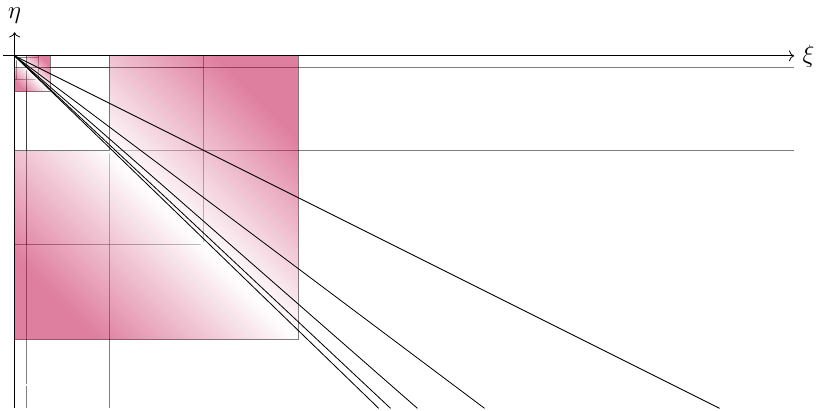}
    \caption{Lacunary cones}
    \label{fig:cones}
\end{figure}

Proposition \ref{prop:2}
is proved in Section \ref{sec:prop:2}. Multipliers vanishing on
the diagonal in Figure \ref{fig:1} play a role as auxiliary objects. We use a positivity of  multipliers
symmetric across the diagonal
to pass to a similar multiplier $m_1$ associated with light blue squares 
but constant on the diagonal.
We then define two further multipliers
$m_2$ and $m_3$ so that $m_1+m_2+m_3$ is constant in the entire plane. This constant multiplier allows
a trivial bound, reducing the estimate for $m_1$ to estimates for $m_2$ and $m_3$.
Multiplier $m_2$ is addressed in Proposition \ref{firststick}.
It is supported near the 
red sticks in Figure \ref{fig:1}. Each stick
is away from the diagonal and has possibly many scales and is therefore not in general Whitney.
However, the multiplier associated with each stick is an elementary tensor and as such telescopes into a small number of scales.
The multiplier $m_1+m_2$ is constant
both on the diagonal as well as on the white L-shaped regions in Figure \ref{fig:cones}.
The multiplier $m_3$ 
is addressed in Proposition 
\ref{firstcz}. 
It is supported
in the at most $J$ purple regions in between
the white L- shaped regions and vanishes on the diagonal. 
Each purple region has a single scale and is Whitney. 
We decompose $m_3$ into a lacunary family of cones towards the diagonal shown in Figure \ref{fig:cones}.
Each lacunary piece is estimated
with Lemma 3 in \cite{MR4171366}. Thanks to vanishing on the diagonal, one has
a geometric sum for the estimates in this family.

Proposition \ref{prop:1} is proved in Section \ref{sec:prop:1}.
We combine the dark blue squares with a  suitable
family of light blue squares with tensor structure
to obtain a multiplier vanishing on the diagonal.
The light blue squares are estimated with Proposition \ref{prop:2}, while the multiplier vanishing on the diagonal is estimated with
Proposition \ref{firstcz}.

Proposition \ref{firststick} is proved in
Section \ref{sec:firststick}. Using the 
off-diagonal property of the multiplier to
preserve crucial cancellation in the innermost integral, we apply the Cauchy-Schwarz inequality 
in the remaining integrals.
We estimate one of the factors on the right-hand side of Cauchy-Schwarz using that the multiplier 
has $J$ summands, which leads to the loss of $J^{\frac 12}$. The other factor we estimate loss-free thanks to the above mentioned cancellation. This loss-free estimate takes the form of a singular Brascamp-Lieb form with datum $D_2$ defined in
\eqref{defd2}. The multiplier $m$ is now three-dimensional,  but consisting of pieces that are naturally of
the form
\begin{equation}\label{2d3d}
    \phi_1(\theta \cdot v_1) \phi_2(\theta \cdot v_2)
\end{equation}
with two vectors $v_1=(0,0,1)$ and $v_2=(1,-1,0)$ as shown in Figure
\ref{3dfigure}. 
\begin{figure}
    \centering
\includegraphics{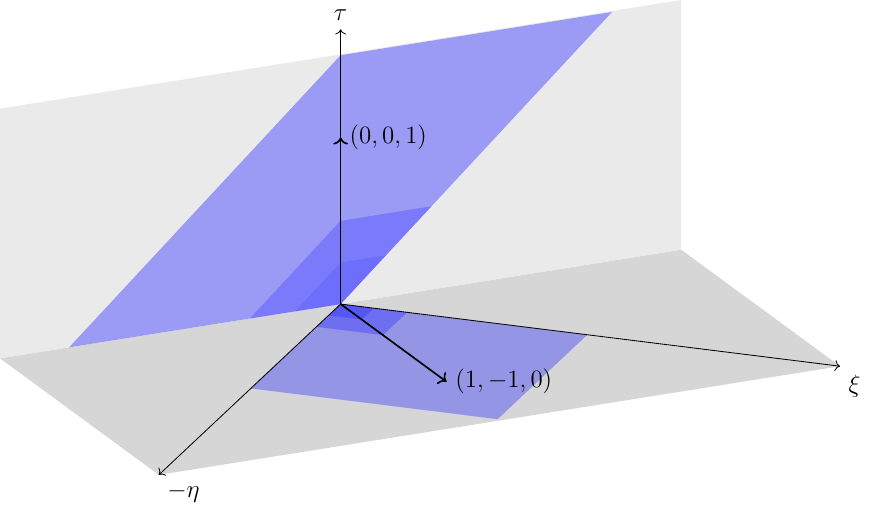}
    \caption{Three-dimensional multiplier}
 \label{3dfigure}
\end{figure}
Typical behaviour of the functions
$\phi_1$ and $\phi_2$ is shown 
in  Figure \ref{3dfigure}
on the planes perpendicular to $v_1$ and $v_2$.
An important role is played by multipliers vanishing on the
critical space spanned by $v_1$ and $v_2$. Such multipliers
are estimated in 
Propositions \ref{secstickphi} and \ref{secstickgauss}
in the non-Whitney case and by Theorem 1.1 in~\cite{MR4510172}
in the Whitney case.
The multiplier $m$ does not vanish on the critical space.
It is first modified using Proposition \ref{secstickphi} towards a  multiplier $m'$ that also 
consists of pieces as in \eqref{2d3d} but is more symmetric
in the variables $\xi,\eta$. The multiplier $m'$ is then estimated by Proposition \ref{thm:4}.

Proposition \ref{thm:4} is proved in Section \ref{sec:thm:4}.
The key to estimating non-Whitney multipliers 
such as $m'$ 
is a new variant of a telescoping identity in this context
that concerns three-dimensional multipliers with a two
dimensional flavor as in Figure \ref{3dfigure} and \eqref{2d3d}. This identity telescopes
a trivial multiplier, the product of the lightest blue  squares minus the product of the darkest
blue squares, into two sums consisting of products of
a square in one plane with a difference of the corresponding square with a consecutive square in the other plane. One sum is arranged to allow a positivity argument
using reflection symmetry across the
diagonal in the $\xi,\eta$ plane, the other sum is arranged similarly to allow a positivity argument with symmetry in 
the shown skew coordinates in the other plane. As the two positive sums add to a trivial multiplier, both are individually bounded. The given multiplier $m'$ can be dominated by one of these constructed multipliers,
using various modifications with Propositions 
 \ref{secstickphi} and \ref{secstickgauss}
and Theorem 1.1 in~\cite{MR4510172}.

 Propositions 
 \ref{secstickphi} and \ref{secstickgauss} are proved in
 Section \ref{sec:secstick}. Vanishing of the multiplier
 on the critical space
 allows a lacunary decomposition away from the critical space 
 and a further Cauchy-Schwarz. This leads to
 singular Brascamp-Lieb estimates with a standard 
 cubical datum and three-fold telescoping identities for
 three-dimensional kernels. The non-Whitney property requires
 telescoping along the scales of the variation sequences.
 There is a mix of discrete telescoping and partial integration, with \ref{secstickphi} more discrete and
 \ref{secstickgauss} more continuous. Analogous but simpler techniques appear in  the Whitney case in Theorem 1.1 in~\cite{MR4510172}.

We have kept the sections past Sections 
\ref{introsection} and \ref{sblsection} independent
of each other, each proves the theorem or one or two propositions and uses some of the other propositions or cited theorems as black boxes.
 
While it is plausible that our approach can be upgraded to an iterative scheme that handles more than three commuting transformations, we decided to complete and circulate the argument in the case of three transformations.
This case has a single 
transition step with the important  new techniques and 
does not appear to involve all
the complications that one expects for the general case.
 \\

 \textit{Acknowledgment.} 
The authors thank Terence Tao for an exchange on the introduction of an earlier version of the paper.
 P.~D.  was partially supported by the NSF grant DMS-2154356. L.~S. was supported by the Primus research programme PRIMUS/21/SCI/002 of Charles University and by Charles University Research program
UNCE/SCI/023. 
C.~T. was funded by the Deutsche Forschungsgemeinschaft (DFG, German Research Foundation) under Germany's Excellence Strategy -- EXC-2047/1 -- 390685813 as well as SFB 1060.
P.~D. and C.~T. were also supported by the NSF grant DMS-1929284 while the authors were in residence at the Institute for Computational and Experimental Research in Mathematics in Providence, RI, during the inspiring Harmonic Analysis and Convexity program.   The authors thank the anonymous referees for their suggestions.

\section{A collection of propositions on singular Brascamp-Lieb forms}
\label{sblsection}
This section contains a number of propositions stating cancellation estimates for 
singular Brascamp-Lieb forms for some data and some 
class of kernels and  with symmetric tuples of test functions. 

The first four propositions and two corollaries share a common datum $D_1$, which, 
after suitable change of variables, arises directly out of the original problem in Theorem~\ref{thm:2}.
Put coordinates \[x=(x_0,x_1,x_2,x_3^0,x_3^1)\] on $\R^5$.
Define 
\begin{equation}\label{defd1}
D_1:=(5,S, \Pi,(\Pi_s)_{s\in S})    
\end{equation}
 with 
$S=\{0,1,2\}\times \{0,1\}$, with $\Pi$ mapping $\R^5$ to $\R^2$ 
as
\[\Pi(x)=(x_3^0-x_0-x_1-x_2, x_3^1-x_0-x_1-x_2),\]
and with $\Pi_s$ for $s=(k,j)$ mapping $\R^5$ to $\R^3$ as
\[\Pi_{(k,j)}(x)=(x_0,x_1,x_2)-x_ke_k+ x_3^j e_k.\]

Each of the three following propositions will have a constant $C$, a 
parameter $J$ and formulate a class of kernels $K$ such that the 
singular Brascamp-Lieb estimate 
\begin{equation}\label{lambdabd}
|\Lambda_{D_1,K} ((f_s)_{s\in S})|\le CJ^{\frac 12}
\end{equation}
holds for all tuples of real-valued Schwartz functions $(f_{s})_{s\in S}$
such that
\begin{equation}\label{topbottom}
f_{(k,0)}=f_{(k,1)}
\end{equation}
for each $k\in \{0,1,2\}$
and 
\begin{equation}\label{normal1}
\|f_{(0,j)}\|_{8}=\|f_{(1,j)}\|_{8} =\|f_{(2,j)}\|_{4} =1
\end{equation}
for each $j\in \{0,1\}$.  We point out that the symmetry assumption~\eqref{topbottom} arises naturally when reducing Theorem~\ref{thm:2} to the propositions stated below. While our arguments could be modified in order to prove these propositions without the extra assumption~\eqref{topbottom}, we decided not to pursue this line of generalization.

Define for any function $\phi$ on $\R^d$ the $L^1$ normalized scaling
\begin{equation*}
\phi_{(t)}(x)=t^{-d} \phi(t^{-1}x).
\end{equation*}
Define the Fourier transform $\widehat{\phi}$ of  $\phi$ 
by integration against the kernel $(x,\xi) \mapsto e^{-2\pi i x \cdot \xi}$.

The kernels in the next proposition satisfy standard two-dimensional symbol estimates with bounds depending on the parameter $k$.
They consist of pieces satisfying a positivity assumption.
Such positivity assumption is used in the proof by adding further
positive terms so as to achieve better behaviour on some frequency diagonal.
The complexity of these kernels is bounded by $J$.

\begin{proposition}[on-diagonal, Whitney, 2D][Proved in Section \ref{sec:prop:1}]\label{prop:1}
Let $\lambda=\frac{3}{2}$.
There exists a constant $C>0$ such that the following holds for all $k \leq 0$. 
Let $J$ be a positive integer and   $(k_j)_{j=1}^J$ 
  a finite strictly monotone increasing sequence of integers.
Let
\[K =\sum_{j=1}^J \Phi_j,\]
where for each $1\le j\le J$ we assume 
$\Phi_j$ is a real valued  function on $\R^2$,
with symmetry $$\Phi_j(u,v)=\Phi_j(v,u)$$ and positivity
in the sense
\begin{equation}\label{poskernel}
\int_{\R^2} f(u)\overline{f(v)}\Phi_j(u,v)\,dudv\ge 0
\end{equation}
for all complex-valued $f$.
We assume further
\begin{equation*}
{\rm supp} (\widehat{\Phi_j})\subset 
([-2^{-k_j+20},-2^{-k_j-20}]\cup [2^{-k_j-20},2^{-k_j+20}])^2
\end{equation*}  
and, for all $(u,v)\in \R^2$,
\begin{equation}\label{E:phi-lambda}
 |(\Phi_j)_{(2^{-k_j})}(u, v)| \leq  2^{\lambda k}(1+2^k|u+v|)^{-10} (1+|u-v|)^{-10}.
\end{equation}
Then estimate \eqref{lambdabd} holds for any tuple as in \eqref{topbottom}, \eqref{normal1}.
\end{proposition}

We note that the particular value $\lambda=\frac{3}{2}$ is not essential for the proof of Proposition~\ref{prop:1}. Evidently, the analogous statement of the proposition becomes stronger for smaller values of $\lambda$. Our proof can be pushed 
to $\lambda>1$ at the expense of allowing the constant in~\eqref{lambdabd} to depend on $\lambda$. On the other hand, the upper bound $\lambda<2$ is needed to apply Proposition~\ref{prop:1} to prove Theorem~\ref{thm:2}.
There are also constants $10$ and $20$ chosen in this proposition which need to be large enough but
also need to relate to similar other constants in other propositions to follow.

If, in the above proposition, each $\Phi_j$ is an elementary tensor of a
suitable function $\phi_j$ with itself, then symmetry and positivity
are automatic, 
and $k$ is naturally chosen as $0$. 
We formulate this as an immediate corollary.

\begin{corollary}
\label{cor:1}
There exists a constant  $C>0$ such that the following holds. 
Let $J$ be a positive integer and   $(k_j)_{j=1}^J$ 
 a finite strictly monotone increasing sequence of integers. Let
\[K =\sum_{j=1}^J \phi_j\otimes \phi_j,\]
where for each $1\le j\le J$ we assume 
$\phi_j$ is a real-valued function on $\R$ with
\begin{equation*}
{\rm supp} (\widehat{\phi_j})\subset 
[-2^{-k_j+20},-2^{-k_j-20}]\cup [2^{-k_j-20},2^{-k_j+20}]
\end{equation*}  
and for all $u\in \R$,
\begin{equation*}
|(\phi_j)_{(2^{-k_j})}(u)| \leq  (1+|u|)^{-20}.
\end{equation*}

Then estimate \eqref{lambdabd} holds for any tuple as in \eqref{topbottom}, \eqref{normal1}.
\end{corollary}

We need the following technical notion of pairs in the next proposition. Let $N=2^{18}$. This large number
 is necessitated by a somewhat
inefficient referral in the proof of Proposition \ref{thm:4} to a theorem in \cite{MR4510172}.
A more hands-on approach should be able to
make this number much more moderate, but this is certainly not important for our argument.
A  {\em $c$-pair} is a pair $(\phi_0,\phi_1)$ of two real valued integrable even functions satisfying the following assumptions. Their Fourier transforms $\widehat{\phi_0},\widehat{\phi_1}$ map
$\R$ to $[0,1]$,
are supported on $[-1,1]$ and constant $1$ on
$[-2^{-1}, 2^{-1}]$, they satisfy 
\begin{equation}\label{windowsum}
(\widehat{\phi_0})^2+(1-\widehat{\phi_1})^2=1,
\end{equation}
 and
\begin{equation*}
\|\widehat{\phi_0}^{(N+30)}\|_\infty , \|\widehat{\phi_1}^{(N+30)}\|_\infty\le c  .   
\end{equation*}
Here and in what follows, $\varphi^{(k)}$ stands for the $k$-th derivative of $\varphi$.
Lemma \ref{cnwindow} below shows that there is $c$ such that
a $c$-pair exists. When $c$ is at most one million
times the infimum of all positive numbers $c'$ such that a
$c'$-pair exists, then $(\phi_0,\phi_1)$ is called a 
universal pair.
A left window is a function $\phi$ such that there exists a function
$\psi$ such that $(\phi,\psi)$ is a universal pair. A right window
is a function $\phi$ such that there exists a function $\rho$
such that $(\rho,\phi)$ is a universal pair.
Note that functions $\phi$ that are both left and right window may exist, but a notion of two sided windows needs caution
as the corresponding functions $\psi$ and $\rho$ may not satisfy 
this notion.

The kernels of the next proposition do not satisfy two-dimensional
symbol estimates, at least not uniformly in the choices of sequences
$k_j$ and $l_j$. They still consist of pieces with a positivity assumption and elementary tensor structure with only two different scales in it and have complexity controlled by $J$.
\begin{proposition}[on-diagonal, non-Whitney, 2D] [Proved in Section \ref{sec:prop:2}]\label{prop:2}
There exists $C>0$ such that the following holds. 
Let $J$ be a positive integer and   $(k_j)_{j=1}^J$ and $(l_j)_{j=1}^J$   two finite sequences of integers  that are interlaced in the sense that $k_{j}+10< l_j$ for $1\le j\le J$ and
$l_j< k_{j+1}$ for $1\le j<J-1$.
Consider a kernel
\[K=\sum_{j=1}^J ({\phi}_{0,j}-{\phi}_{1,j})\otimes  ({\phi}_{0,j}-{\phi}_{1,j}),\]
where, for each $j$,
${(\phi_{0,j})_{(2^{-k_j})}}$ is a left window 
and ${(\phi_{1,j})_{(2^{-l_j})}}$ is a right window.

Then estimate \eqref{lambdabd} holds for any tuple as in \eqref{topbottom}, \eqref{normal1}.
\end{proposition}

Using Corollary \ref{cor:1}, we have the following 
Corollary of Proposition \ref{prop:2}, 
\begin{corollary}
\label{cor:2}
The variant of Proposition \ref{prop:2}, where
the assumption $k_j+10<l_j$ is replaced by the assumption $k_j<l_j$, holds.
\end{corollary}
To see this corollary, we split the sequence into terms with
$k_j+10 \geq l_j$ and $k_j+10<l_j$. The former terms are estimated with
 Corollary~\ref{cor:1}, while the latter are estimated with
Proposition \ref{prop:2}.

In contrast to the last proposition, the kernel of the next proposition does not
oscillate on the critical frequency diagonal $\xi+\eta=0$. The complexity still is controlled by  $J$. We no longer have the positivity assumptions, but  we do satisfy standard symbol estimates, 
with bounds depending on the parameter $k$.

\begin{prop}[off-diagonal, Whitney, 2D] [Proved in Section \ref{sec:firstcz}] \label{firstcz} 
Let $\lambda=\frac{3}{2}$. There exists a constant $C>0$  such that the following holds for all $k\leq 0$. Let $J$ be a positive integer and let $(k_j)_{j=1}^J$ be a finite 
strictly increasing sequence of integers. Let $(\Phi_j)_{j=1}^J$ be a finite sequence of real valued functions on $\R^2$.
Assume that 
\[\textup{supp}(\widehat{\Phi_j}) \subseteq \{(\xi,\eta)\in \R^2 : 2^{-k_{j}-30}\leq |(\xi,\eta)| \leq 2^{-k_j+30} \}.  \]
Assume further that  for all $(u,v)\in \R^2$, 
\begin{equation}\label{E:phi-lambda3/2}
|(\Phi_j)_{(2^{-k_j})}(u,v)| \leq  2^{\lambda k}(1+2^k|u+v|)^{-4} (1+|u-v|)^{-4} 
+(1+|u+v|)^{-4} (1+|u-v|)^{-4}.
\end{equation}
Let $K$ be defined by
\[
K=\sum_{j=1}^J \Phi_{j}
\]
and assume that $\widehat{K}$ vanishes on the diagonal $\{(\xi,\eta)\in \R^2: \xi+\eta=0\}$.

Then estimate \eqref{lambdabd} holds for any tuple as in \eqref{topbottom}, \eqref{normal1}.
\end{prop}

The kernel of the next proposition also vanishes on the critical diagonal. It does not satisfy standard two-dimensional symbol
estimates uniformly in $k_j$. It has no positivity assumption,
but similarly to some of the positive kernels above it 
is a sum of $J$ tensors 
with few scales in it.

\begin{prop} [off-diagonal, non-Whitney, 2D][Proved in Section \ref{sec:firststick}]\label{firststick} 
There exists a constant $C>0$ such that the following holds.
Let $J$ be a positive integer and $(k_j)_{j=0}^J$   a   finite increasing sequence of integers with $k_{j-1}+10 \le k_{j}$
for $1\le j \leq J$. 
For $1\le j\le J$,  let $\phi_{0,j}$, $\phi_{1,j}$, $\phi_{2,j}$
be functions such that
${(\phi_{0,j})_{(2^{-k_{j-1}})}}$ is a left window, while
${(\phi_{1,j})_{(2^{-k_j})}}$  and 
${(\phi_{2,j})_{(2^{4-k_j})}}$ are right windows.
Define 
\begin{equation}\label{kerfirststick}
K = \sum_{j=1}^J (\phi_{0,j}-\phi_{2,j})\otimes  \phi_{1,j} .  
\end{equation} 

Then estimate \eqref{lambdabd} holds for any tuple as in \eqref{topbottom}, \eqref{normal1}.
 \end{prop}

The remaining propositions share a  singular  Brascamp-Lieb datum $D_2$.
The datum $D_2$ arises as a reduction from $D_1$ after a  Cauchy-Schwarz inequality.
Put coordinates  $x=(x_0,x_1,x_{2}^0,x_{3}^0,x_{2}^1,x_{3}^1)$ on $\R^6$.
Define 
\begin{equation}\label{defd2}
    D_2:=(6,S, \Pi,(\Pi_s)_{s\in S})
\end{equation} with 
$S=\{0,1\}\times {\mathcal C}$,  where
${\mathcal C}$ is the set of functions $j:\{0,1\}\to \{0,1\}$,  with $\Pi$ mapping $\R^6$ to $\R^3$ as
\[\Pi(x) = (x_2^0-x_0-x_1-x_3^0, x_2^1-x_0-x_1-x_3^0, x_3^1-x_3^0),\]
and with $\Pi_s$ for $s=(k,j)$ mapping $\R^6$ to $\R^3$ as
\[\Pi_{(k,j)}(x)= (x_k, x_2^{j(0)}, x_3^{j(1)}). \]

For this datum $D_2$ and a kernel $K$, we are interested in a    loss free estimate 
\begin{equation}\label{lambdabd1}
|\Lambda_{D_2,K} ((f_s)_{s\in S})|\le C 
\end{equation}
for any tuple of real-valued Schwartz functions $(f_s)_{s\in S}$ with 
\begin{equation}\label{twosides}
f_{(k,j)}=f_{(k,j')} 
\end{equation}
for all $k\in\{0,1\}$ and $j,j'\in \mathcal{C}$, and
\begin{equation}\label{normal2}
\|f_{s}\|_{8}=1
\end{equation}
for all $s\in S$.

The next proposition is a variant of Proposition \ref{prop:2},
adjusted to the datum $D_2$. The kernel has some positivity properties and pieces arising from suitable elementary tensor structure. The complexity $J$ here is not relevant, as we obtain estimates independent of $J$.

We write $g$ for the Gaussian 
$g(x)=e^{-\pi |x|^2}$,
typically in one dimension but occasionally in more than one dimension.
We have $\widehat{g}=g$. We write $h$ for the derivative of the Gaussian in one dimension, $h(x)=-2\pi x g(x)$. Recall $N=2^{18}$.

 \begin{proposition}[on-critical, non-Whitney, 3D][Proved in Section \ref{sec:thm:4}]\label{thm:4}
There exists a constant $C>0$ such that the following holds.
Let $\alpha\ge 1$. Let $J$ be a positive integer and $(k_j)_{j=0}^J$   a   finite increasing sequence of integers with $k_{j-1}+10 \leq k_{j}$ for $1\le j \leq J$. Let $(m_j)_{j=1}^J$ be a sequence of real numbers with
$k_j-1\le m_j\le k_j$ for $1\le j\le J$.  For $0\le j\le J$, let $\chi_j$
be a function such that ${(\chi_j)_{(2^{2-k_j})}}$ is a left window
and let $\phi_{j}$ be such that $\wh{\phi_j} \geq 0$ and
$$ (\widehat{\phi_{j}})^2=(\widehat{\chi_{j-1}})^2-
(\widehat{\chi_{j}})^2.$$
Let
\begin{equation*}
K(u,v,z)=\alpha^{-N} \sum_{j=1}^J \int_{\R}g_{(\alpha 2^{m_j})}(u+p) g_{(\alpha 2^{m_j})}(v+p)
 {\phi}_{j}(z+p){\phi}_{j}(p)  \, dp.
\end{equation*}

Then estimate \eqref{lambdabd1} holds for any tuple as in \eqref{twosides}, \eqref{normal2}.
\end{proposition}

Proposition \ref{thm:4} will be proven using the next two propositions.
Both involve the datum $D_2$. Both exploit
a vanishing of the function $\widehat{K}$
on the critical space $\xi+\eta=0$.

\begin{prop}[off-critical, non-Whitney, 3D, discrete][Proved in Section \ref{sec:secstick}]
\label{secstickphi} 
 There is a constant $C$ such that the following holds. Let $J$ be a positive integer. 
For $1\le i\le 2$,  let  $(a_{i,j})_{j=1}^{J}$  be increasing sequences of positive real numbers.

For $1\leq j \leq J$,  let $\rho_{j}:\R^4\to \R$ be a continuous function satisfying 

\begin{equation}\label{rhobound}
\int_{\R^2} |\rho_{j} | (u_1+p,u_2+p,u_3+r,u_4+r)\, dp dr
 \le a_{1,j}^{-1}(1+a_{1,j}^{-1}|u_1-u_2|)^{-2}
 a_{2,j}^{-1}(1+a_{2,j}^{-1}|u_3-u_4|)^{-2}
\end{equation}
for every $(u_1,u_2,u_3,u_4)\in \R^4$.
Let $(c_j)_{j=0}^J$ be an increasing sequence of positive {real} numbers,
well separated in that $2c_{j-1}\le c_j$ for $1\le j\le J$.
 Let $\chi$ be a left window.
For $1\le j\le J$ let $\phi_{j}:\R\to \R$ be   a continuous function, which exists due
 to the left window property of $\chi$, satisfying  $\wh{\phi_j} \geq 0$  and 
 \[(\wh{\phi_{j}})^2 = (\wh{{\chi}_{(c_{j-1})}})^2 
 - (\wh{{\chi}_{(c_j)}})^2. \]
 Let $K$ be defined by   
      \begin{equation}
        \label{ker1space}
    {K}(u,v,z)= \sum_{j=1}^J \int_{\R^3}{\phi_{j}}(p)
    {\phi_{j}}(q) {\rho_{j}}(u+p+q+r,v+p+q+r,z+r,r) \,dpdqdr. 
    \end{equation}

      Then estimate \eqref{lambdabd1} holds for any tuple as in \eqref{twosides}, \eqref{normal2}.
 \end{prop}

The orthogonal complement $V^\perp$ of the subspace \[V=\{(\xi,\eta,\tau, -(\xi+\eta+\tau), -(\xi+\eta), -(\xi+\eta)): \xi,\eta,  \tau \in \R\},\] 
 of $\R^6$ can be parameterized as  
\[\{(p+q+r,p+q+r,r,r,p,q): p,q,r\in \R\}.\]
As \eqref{ker1space} is an integral over $V^\perp$ of a function $F$ in $\R^6$,
its Fourier transform is the restriction to $V$ of the Fourier transform
of $\widehat{F}$ to that subspace. Hence, for some universal constant $C$,
     \begin{equation}
        \label{ker1}
    \widehat{K}(\xi,\eta,\tau)= C\sum_{j=1}^J \widehat{\phi_{j}}(\xi+\eta)^2 \widehat{\rho_{j}}(\xi,\eta,\tau,-\xi-\eta-\tau). 
    \end{equation}
 This expression shows the vanishing of $\widehat{K}(\xi,\eta,\tau)$
on the hyperplane $\xi+\eta=0$.

Also in the following proposition, $\widehat{K}$ vanishes on $\xi+\eta$.
It is made up by a very specific part in the variables $\xi,\eta$ and a 
rather general part in the variables $\tau$ and $\tau+\xi+\eta$.
 
\begin{prop}[off-critical, non-Whitney, 3D, continuous][Proved in Section \ref{sec:secstick}]
\label{secstickgauss}
 There is a constant $C$ such that the following holds. Let $J$ be a positive integer  and $(a_j)_{j=0}^J$, $(b_j)_{j=1}^J$  be increasing sequences of positive real numbers.    For $1\le j\le J$
 let $\phi_{j} :\R^2\to \R$ be a continuous function satisfying  
\begin{equation}\label{phistickgauss}
|\phi_{j}(u_1,u_2)| \leq (b_j)^{-2}(1+b_j^{-1}|(u_1, u_{2})|)^{-4} .  
\end{equation}
Let $K$ be a kernel such that
    \begin{equation}
        \label{ker2}
    \widehat{K}(\xi,\eta,\tau) =    \sum_{j=1}^J \int_{{a_{j-1}}}^{{a_{j}}}  t^2(\xi+\eta)^2 g(t\xi) g(t\eta)  \,\frac{dt}{t} \,
    \widehat{\phi_{j}}(\tau,-\xi-\eta-\tau).
    \end{equation}
    
    Then estimate \eqref{lambdabd1} holds for any tuple as in \eqref{twosides}, \eqref{normal2}.

\end{prop}

We remark on a symmetry in the datum $D_2$. We do a change of variables in the kernel using the linear map 
\[L(a,b,c)= (a+b-c, a-b,c).\] 
Define
\[\widetilde{\Pi}(x):=L\circ \Pi(x)=
(x_2^{0}+x_2^{1}-x_3^{0}-x_3^1-2(x_0+x_1), x_2^0-x_2^1, x_3^1-x_3^0).\]
Define $\widetilde{D}_2$ from $D_2$ by replacing $\Pi$ by $\widetilde{\Pi}$ and choose $\widetilde{K}$ so that $\widetilde{K}\circ L=K$. We obtain 
$$\Lambda_{D_2,K}((f_s)_{s\in S})=\Lambda_{\widetilde{D}_2,\widetilde{K}}(({f}_s)_{s\in S}).$$
The map $\widetilde{\Pi}$ has a symmetry under interchanging the 
last two entries at the same time as precomposing with the involution
$$(x_0,x_1,x_2^{0},x_3^{0},x_2^{1},x_3^{1})\mapsto (x_0,x_1,-x_3^{0},-x_2^{0},-x_3^{1},-x_2^{1})$$
This involution can be seen as acting on the tuple of functions $f_s$,
and hence we have the following consequence for the associated form.
Define $\widetilde{K}^*(a,b,c)=\widetilde{K}(a,c,b)$. For $j\in \mathcal{C}$, define $j^*\in \mathcal{C}$ by $j^*(l)=j(1-l)$ and define
$f^*_{(k,j)}(a,b,c)=f_{(k,j^*)}(a,-c,-b)$. Then
\begin{equation}\label{hiddensymm}
\Lambda_{\widetilde{D}_2,\widetilde{K}}(({f}_s)_{s\in S})
=\Lambda_{\widetilde{D}_2,\widetilde{K}^*}(({f}_s^*)_{s\in S}).
\end{equation}

We finally introduce a further datum $D_A$, which is 
  associated with a   regular $3\times 3$ matrix  $A$ and has $n=6$.   Let  $S$ be the set of functions $S:\{0,1,2\} \to \{0,1\}$.  We put coordinates  $x=(x_1^0,x_2^0,x_{3}^0,x_{1}^1,x_{2}^1,x_{3}^1)$ on $\R^6$. 
We define
\begin{equation} \label{datumDA}
    D_A:=(6,S, \Pi,(\Pi_s)_{s\in S}),
\end{equation}
 where 
the projection 
$\Pi:\R^6\to \R^3$ is given by 
\begin{equation}
    \label{piT}
    \Pi(x)^T = (I,A) x^T,
\end{equation}
 where  $I$ is the $3\times 3$ identity matrix and $(I,A)$ is a $3 \times 6$ block matrix. For $s\in S$,   $\Pi_s: \R^6 \to \R^3$ is given by
\[\Pi_{s}(x)= (x_1^{s(0)}, x_2^{s(1)}, x_3^{s(2)}). \]
Note that after the relabelling of the coordinates, this datum has the same components as the datum $D_2$ except for the choice of the projection $\Pi$. We have used transposes in \eqref{piT} as we usually write vectors as rows while  the matrix equation  \eqref{piT} expects columns.
The   datum $D_A$  will 
  be used   the proofs of Propositions \ref{firstcz}, \ref{secstickphi}, and \ref{secstickgauss}. In the latter two cases we will only use it with $A=-I$.

We conclude this section with the previously announced  existence result.

\begin{lemma}\label{cnwindow} 
There exists a $c>0$ and a $c$-pair as defined near \eqref{windowsum}. 
\end{lemma}

\begin{proof} 

Let $\psi :[0,\infty) \to \R$ be a smooth monotone decreasing function  with
\[\psi(x)=1/2 \ \ {\rm for}\ \  x\in [0, 5/6],\] 
\[\psi(x)=0 \ \ {\rm for}\ \  x\in [1, \infty).\]
Let $\rho:[0,\infty) \to \R$ be a smooth monotone increasing function with
\[\rho(x)=0 \ \ {\rm for}\ \  x\in [0, 1/2],\] 
\[\rho(x)=3^{1/2} \ \ {\rm for}\ \  x\in [4/6, \infty).\]
There exists a smooth even function $\phi_0$ on $\R$
such that its Fourier transform is nonnegative and satisfies on $[0,\infty)$
\begin{equation*}
(\wh{\phi_0})^2= (4-\rho^2)\psi^2,\end{equation*}
because the right-hand side equals $\psi^2$ on $[4/6,\infty)$
and is bounded below by $1/4$ on $[0,5/6]$ and constant one 
on $[0,1/2)$.
There exists a smooth even function $\phi_1$ on $\R$ such that its Fourier transform is nonnegative and fulfills on the interval $[0,\infty)$
\begin{equation*}
(1-\wh{\phi_1})^2=1-(4 -\rho^2)\psi^2,\end{equation*}
because the right-hand side equals $\rho^2/4$ on $[0,5/6]$
and is bounded below by $3/4$ on $[4/6,\infty)$ and constant
on $[0,1/2]$.
The pair $(\phi_0,\phi_1)$ then satisfies the assumptions 
for a $c$-pair with
$$c=\max(\|{\wh{\phi_0}}^{(N+30)}\|_\infty,  \|{\wh{\phi_1}}^{(N+30)}\|_\infty).$$
This proves the lemma.
\end{proof}

We write $A\lesssim B$ if there exists a constant $C>0$ such that
$|A|\le C|B|$ uniformly over all values of parameters appearing in the expressions $A$ and $B$.

\section{Proof of Theorem \ref{thm:2} from Proposition~\ref{prop:1} and Corollaries~\ref{cor:1} and~\ref{cor:2}}
\label{proofthm:2}

This section follows the corresponding argument in \cite{MR3904183} for two commuting transformations
with minor modifications.
We summarize and streamline the argument.

Let $J$ be given, without loss of generality we may assume $J>2$. Let also positive real numbers $t_0<t_1\cdots < t_J$ be given.  
Let  $f_0,f_1,f_2$ be real valued measurable functions on $\R^3$, normalized as 
\[\|f_0\|_{4} =  \|f_1\|_{8}=\|f_2\|_{8}=1. \]
We will prove a weak-type  endpoint estimate at $r=4$, namely
for any $f_0\in L^4(\R^3)$  and $f_1,f_2\in {L}^8(\mathbb{R}^3)$ with respective norms one,
\begin{equation}\label{eq:doublenormvar}
\sum_{j=1}^{J} \|M_{t_j}(f_0,f_1,f_2) - M_{t_{j-1}}(f_0,f_1,f_2)\|_{2}^2 \lesssim J^{\frac 12}.
\end{equation}

We call  \eqref{eq:doublenormvar} an endpoint estimate as it would follow from the hypothetical
inequality \eqref{e:main2} with $r=4$ by the Cauchy-Schwarz inequality, and conversely
\eqref{eq:doublenormvar} implies \eqref{e:main2} for parameters $r>4$.
Namely, \eqref{eq:doublenormvar} allows by Chebyshev's inequality to estimate the number of 
$\lambda$-jumps of the norm by $O(\lambda^{-4})$, which then allows to deduce \eqref{e:main2} by a layer cake representation of the $r$-variation. 
Theorem \ref{thm:2} will thus follow as soon as we prove \eqref{eq:doublenormvar}.

We decompose the characteristic function $\mathbbm{1}_{[0,1)}$ into smoother functions.
Let $\chi$ be a left window and define  
\[{\theta}:={\chi}-{\chi}_{(2)}.\]
Then $\widehat{\theta}$ is supported in $[-1,-2^{-2}]\cup [2^{-2},1]$ and, as detailed
in \cite[Section 2.4]{MR3904183},
\[
\mathbbm{1}_{[0,1)} = \mathbbm{1}_{[0,1)}\ast\chi + 
\sum_{k=-\infty}^{-1} \mathbbm{1}_{[0,\infty)}\ast\theta_{(2^k)}
-\sum_{k=-\infty}^{-1} \mathbbm{1}_{[1,\infty)}\ast\theta_{(2^k)}
\]
\begin{equation}\label{eq:seriesexpansion}
=: \varphi+ \sum_{k=-\infty}^{-1} \varphi_{0,k}+ \sum_{k=-\infty}^{-1}\varphi_{1,k}.
\end{equation}

For  $\vartheta\in L^1(\R)$  we define in analogy with \eqref{eq:averages} for $x \in \R^3$
\begin{equation*}
M_{t}^\vartheta(f_0,f_1,f_2)(x) :=  \int_{\R} f_0(x+ue_0) f_1(x+ue_1) f_2(x+ue_2)\, \vartheta_{(t)}(u)\,du.
\end{equation*}

Using \eqref{eq:seriesexpansion} and  the triangle inequality on the sum in $k$,
it suffices to show in place of
\eqref{e:main2} for every $k\le -1$, 
\begin{equation}\label{mainbump} \sum_{j=1}^{J} \|M^{\varphi}_{t_j}(f_0,f_1,f_2) - M^{\varphi}_{t_{j-1}}(f_0,f_1,f_2)\|_{2}^2  \lesssim J^{\frac 12} ,\end{equation}
\begin{equation}\label{leftbump}
\sum_{j=1}^{J} \| M^{\varphi_{0,k}}_{t_j}(f_0,f_1,f_2) - M^{\varphi_{0,k}}_{t_{j-1}}(f_0,f_1,f_2)\|_{2}^2  \lesssim 2^{2k} J^{\frac 12}, \end{equation}
\begin{equation}\label{leftbump1}
\sum_{j=1}^{J} \| M^{\varphi_{1,k}}_{t_j}(f_0,f_1,f_2) - M^{\varphi_{1,k}}_{t_{j-1}}(f_0,f_1,f_2)\|_{2}^2  \lesssim 2^{\gamma k} J^{\frac 12}, \end{equation}
where $\gamma =\frac{1}{2}$. In fact, it will follow from our argument that inequality~\eqref{leftbump1} continues to hold with any $\gamma<1$, at the expense of allowing the constant in that inequality to depend on $\gamma$.
The estimate \eqref{mainbump} is acceptable and the estimates \eqref{leftbump} and
\eqref{leftbump1} give a geometric series over $k\leq -1$ and are thus acceptable as well.

We first prove \eqref{mainbump}. 
We reduce further \eqref{mainbump} to the analogous estimate but with the bump function $\varphi$ replaced by one whose Fourier transform is constant near the origin. 
We write
\begin{equation*}
\varphi = \chi + (\varphi - \chi) = \chi + \sum_{l=-2}^\infty (\varphi-\chi)\ast\theta_{(2^l)}
=: \chi + \sum_{l=-2}^\infty \varphi_{2,l}.
\end{equation*}

It then suffices to show
  \begin{equation}
      \label{long1}
      \sum_{j=1}^{J} \|M^{\chi}_{t_j}(f_0,f_1,f_2) - M^{\chi}_{t_{j-1}}(f_0,f_1,f_2)\|_{2}^2  \lesssim J^{\frac 12},
  \end{equation}
  \begin{equation}
      \label{long2}
      \sum_{j=1}^{J} \|M^{{\varphi_{2,l}}}_{t_j}(f_0,f_1,f_2) - M^{{\varphi_{2,l}}}_{t_{j-1}}(f_0,f_1,f_2)\|_{2}^2  \lesssim
      2^{-2l}J^{\frac 12}.
  \end{equation}

 We first prove \eqref{long1}. We split into long and short variation as in \cite{MR2434308}.
 Enlarging the sequence $t_j$ 
if necessary while at most doubling the number of terms and retaining at least a quarter of the left-hand side of \eqref{long1}, we may assume that for each $t_j$ there is a $t_i$ which is an integer  power of two with $t_i\le t_j<2t_i$.
Let $(k_i)_{i=0}^I$ be the increasing sequence of all $k_i$ such that 
the power $2^{k_i}$ occurs in the sequence $(t_j)_{j=1}^J$. We have $I\le J$.
It then suffices to show the short and long variation bounds 
\begin{equation}\label{mainbumpshort} \sum_{i=0}^I\,\sum_{j:\, 2^{k_i}< t_j\le 2^{k_{i}+1}} \|M^{\chi}_{t_{j}}(f_0,f_1,f_2) - M^{\chi}_{t_{j-1}}(f_0,f_1,f_2)\|_{2}^2  \lesssim J^{\frac 12} ,\end{equation}
\begin{equation}\label{mainbumplong} \sum_{i=1}^{I} \|M^{\chi}_{2^{k_i}}(f_0,f_1,f_2) - M^{\chi}_{2^{k_{i-1}}}(f_0,f_1,f_2)\|_{2}^2  \lesssim J^{\frac 12} .\end{equation}

We first discuss the short variation \eqref{mainbumpshort}. 
 We denote $T\chi(s):= (s\chi(s))'$, so that   
 \[(T\chi)_{(t)}(s)=-t\partial_t (\chi_{(t)}(s)),\]
and  we will use $T$ throughout the section.
By the fundamental theorem of calculus and the Cauchy-Schwarz inequality, we have for
 $x\in \R^3$
and every $1\leq i \leq I$, 
\[  \sum_{j:\,2^{k_i}< t_j\le 2^{k_i+1}} |M^{\chi}_{t_j}(f_0,f_1,f_2)(x)-M^{\chi}_{t_{j-1}}(f_0,f_1,f_2)(x)|^2 \leq   \int_{1}^2 \big( M_{2^{k_i}t}^{T\chi}(f_0,f_1,f_2)(x)\big)^2 \,\frac{dt}{t}  . \]
It then suffices to show
\begin{equation*}
      \sum_{i=0}^I\int_{\mathbb{R}^3} \int_{1}^2 \big( M_{2^{k_i}t}^{T\chi}(f_0,f_1,f_2)(x) \big)^2\, \frac{dt}{t} dx \lesssim J^{\frac 12}.
\end{equation*}
Expanding the square and moving the integral in $t$ outside,   the left-hand side   becomes
\begin{equation}\label{expandlhs}
\int_1^2\sum_{i=0}^I  \int_{\mathbb{R}^5} \Big[\prod_{n=0}^2 f_n(x+ue_n) 
f_n(x+ve_n) \Big] (T\chi)_{(2^{k_i}t)} (u) (T\chi)_{(2^{k_i}t)} (v) 
  \,dx du dv  \frac{dt}{t}.
\end{equation}

The expression \eqref{expandlhs} takes the form  
\begin{equation}
    \label{longvarcomplsq}
    \int_1^2 \Lambda_{D_1,{K_t}}((f_s)_{s\in S})\, \frac{dt}{t},
\end{equation}
where for $s=(k,j)\in \{0,1,2\}\times \{0,1\}$  and $y=(y_0,y_1,y_2)$ we have set
\begin{equation}\label{E:f}
{f}_{s}(y)=f_k(y-(y_0+y_1+y_2-y_k)e_k )
\end{equation}
and 
\begin{equation}\label{ktuv}
K_t(u,v) := \sum_{i=0}^I (T\chi)_{(2^{k_i}t)}(u) (T\chi)_{(2^{k_i}t)}(v).
\end{equation}
Indeed, writing $x=(x_0,x_1,x_2)$ and changing variables
\begin{equation}\label{E:change1}
u=x_3^0-x_0-x_1-x_2
\end{equation}
\begin{equation}\label{E:change2}
v=x_3^1-x_0-x_1-x_2,
\end{equation}
we obtain with the projections $\Pi_s$ of the datum $D_1$, 
\[{f}_{(k,0)}(\Pi_{(k,0)}(x,x_3^0, x_3^1))=f_k(x+ue_k)\]
\[{f}_{(k,1)}(\Pi_{(k,1)}(x,x_3^0, x_3^1))=f_k(x+ve_k)\]

It suffices to prove bounds uniformly for fixed $t\in [1,2]$ on the integrand of \eqref{longvarcomplsq}.
For this we apply Corollary \ref{cor:1}  with 
the sequence $(k_i)_{i=0}^I$ and $\phi_i$ suitable multiples of $(T\chi)_{(2^{k_i}t)}$ and use
\begin{equation}\label{psisupp}
    {\rm supp}(\widehat{T\chi})\subset    [-1,-2^{-1}]\cup[2^{-1},1] ,
\end{equation}
\begin{equation}
    \label{psidecay}
    |T\chi (u)| \lesssim  (1+|u|)^{-20}.
\end{equation}
This proves \eqref{mainbumpshort}.

 Next, we prove  the long variation bound  \eqref{mainbumplong}.
Recalling the universal pair $(\chi,\phi)$, by the triangle inequality,
it suffices to show
\begin{equation}\label{mainbumplong1} \sum_{i=1}^{I} \|
M^{\chi}_{2^{k_{i-1}}}(f_0,f_1,f_2)-M^{\phi}_{2^{k_i}}(f_0,f_1,f_2) \|_{2}^2  \lesssim J^{\frac 12},\end{equation}
\begin{equation}\label{mainbumplong2} \sum_{i=1}^{I} \|M^{\chi}_{2^{k_i}}(f_0,f_1,f_2) - M^{\phi}_{2^{k_{i}}}(f_0,f_1,f_2)\|_{2}^2  \lesssim J^{\frac 12} .\end{equation}
We first prove \eqref{mainbumplong1}.
We expand out the square of the $L^2$ norm to reduce matters to estimating 
   \begin{equation}\label{longexpl}
   \sum_{i=1}^I  \int_{\mathbb{R}^5} \Big[\prod_{n=0}^2 f_n(x+ue_n) 
f_n(x+ve_n) \Big] 
({\chi}_{(2^{k_{i-1}})}-{\phi}_{(2^{k_{i}})})(u) ({\chi}_{(2^{k_{i-1}})}-{\phi}_{(2^{k_{i}})})(v)
 \,dx du dv. \end{equation} 
 Performing the same change of variables in the Brascamp-Lieb datum as in \eqref{expandlhs}, we rewrite it as
  \begin{equation}
      \label{firstdatum-reduction}
      \Lambda_{D_1,K}((f_s)_{s\in S})
  \end{equation}
with 
$$K(u,v)=\sum_{i=1}^I ({\chi}_{(2^{k_{i-1}})}-{\phi}_{(2^{k_{i}})})(u) ({\chi}_{(2^{k_{i-1}})}-{\phi}_{(2^{k_{i}})})(v).$$
We estimate this with Corollary \ref{cor:2} of Propositions \ref{prop:1} and \ref{prop:2}, using that $\widehat{\chi}$ is a left window
and $\widehat{\phi}$ is a right window, and after splitting the sum into
even and odd indices $j$ to assure spacing of the sequences $k_j$ and $l_j$.
This completes the discussion of \eqref{mainbumplong1}.  
Similarly, estimating \eqref{mainbumplong2} reduces to estimating
 a form \eqref{firstdatum-reduction} with kernel
$$K(u,v)=\sum_{i=1}^I ({\chi}_{(2^{k_{i}})}-{\phi}_{(2^{k_{i}})})(u) ({\chi}_{(2^{k_{i}})}-{\phi}_{(2^{k_{i}})})(v).$$
This is done with Corollary \ref{cor:1}.
 This completes the discussion of \eqref{mainbumplong2}
and thus the discussion of \eqref{long1}.

 Next, we consider the decaying lacunary pieces near the origin \eqref{long2}. We define
 \[\varphi_{3,l}(x):= 2^{l}(\varphi_{2,l})_{(2^{-l})}(x)\] 
 and we replace $t_j$ by $2^l t_j$, using that the sequence $t_j$ was arbitrary,
 to turn \eqref{long2} into
  \begin{equation}
      \label{long3}
      \sum_{j=1}^{J} \|M^{{\varphi_{3,l}}}_{t_j}(f_0,f_1,f_2) - M^{{\varphi_{3,l}}}_{t_{j-1}}(f_0,f_1,f_2)\|_{2}^2  \lesssim  J^{\frac 12}.
  \end{equation} 
 Analogously to our discussion of \eqref{long1}, we pass to short and long variation.
 The short variation we estimate analogously using in place of \eqref{psisupp}
 and \eqref{psidecay}  
 \begin{equation}\label{varphisupp}
    {\rm supp}(\widehat{T\varphi_{3,l}})\subset    [-1,-2^{-1}]\cup[2^{-1},1]
\end{equation}
\begin{equation}
    \label{varphidecay}
    |T\varphi_{3,l}(u)| \lesssim  (1+|u|)^{-20}, 
\end{equation}
which follows because  $\widehat{\varphi}-\widehat{\chi}$ vanishes at the origin.
This completes the estimate for the short variation.

The long variation we expand similarly as  \eqref{longexpl} above into
\begin{equation}
    \label{longvar-lambda}
    \Lambda(f_0,f_1,f_2):= 
 \sum_{i=1}^I  \int_{\mathbb{R}^5} \Big[\prod_{n=0}^2 f_n(x+ue_n) 
f_n(x+ve_n) \Big]  \end{equation}
\[\times({(\varphi_{3,l})}_{(2^{k_{i-1}})}-{(\varphi_{3,l})}_{(2^{k_{i}})})(u) ({(\varphi_{3,l})}_{(2^{k_{i-1}})}-{(\varphi_{3,l})}_{(2^{k_{i}})})(v)
 \,dx du dv. \] 
By the distributive law, \eqref{longvar-lambda} is the difference of the two terms of the form
\begin{equation}\label{twomterms}
\sum_{i=1}^I  \int_{\mathbb{R}^5} \Big[\prod_{n=0}^2 f_n(x+ue_n) 
f_n(x+ve_n) \Big]  {(\varphi_{3,l})}_{(2^{m_{i}})}(u) ({(\varphi_{3,l})}_{(2^{k_{i-1}})}-{(\varphi_{3,l})}_{(2^{k_{i}})})(v)
 \,dx du dv,
 \end{equation}
 with $m_i = k_i$ and with $m_i = k_{i-1}$, respectively.
We write \eqref{twomterms} as
\[\sum_{i=1}^I  \int_{\mathbb{R}^3} \Big[\int_{\R}\prod_{n=0}^2 f_n(x+ue_n) {(\varphi_{3,l})}_{(2^{m_{i}})}(u)\,du
  \Big] \]
\[ \times \Big[\int_{\R}\prod_{n=0}^2  
f_n(x+ve_n)  ({(\varphi_{3,l})}_{(2^{k_{i-1}})}-{(\varphi_{3,l})}_{(2^{k_{i}})})(v)\,dv\Big] 
 \,dx \]
and apply the Cauchy-Schwarz inequality in $x$ and in the summation. This gives
\[\Lambda(f_0,f_1,f_2)\leq \widetilde{\Lambda}(f_0,f_1,f_2)^{\frac 12}\Lambda(f_0,f_1,f_2)^{\frac 12}\]
with
\begin{equation}\label{tildelambda}
\widetilde{\Lambda}(f_0,f_1,f_2)=
 \end{equation}
 \[\Big[ \sum_{i=1}^I  \int_{\mathbb{R}^5} \Big[ \prod_{n=0}^2 f_n(x+ue_n) f_n(x+ve_n) \Big]  {(\varphi_{3,l})}_{(2^{m_{i}})}(u){(\varphi_{3,l})}_{(2^{m_{i}})}(v)\,dxdudv
   \Big]^{\frac 12}. \] 
By bootstrapping, it suffices to prove a bound on $\widetilde{\Lambda}(f_0,f_1,f_2)$
in place of $\Lambda(f_0,f_1,f_2)$.
This should be compared with the integrand in \eqref{expandlhs} for fixed $t$.
By the same change of variables as there, \eqref{tildelambda} equals $\Lambda_{D_1,K}((f_s)_{s\in S})$ with 
\begin{equation}
    \label{longvar-lambda-ker}
    K(u,v) = \sum_{i=1}^I {(\varphi_{3,l})}_{(2^{m_{i}})}(u){(\varphi_{3,l})}_{(2^{m_{i}})}(v). 
\end{equation}
Applying Corollary \ref{cor:1} of Proposition \ref{prop:1} yields a bound for this term and 
    finishes the proof of  \eqref{long2}.  The assumptions of Corollary \ref{cor:1} are satisfied, which can be verified similarly as inequalities \eqref{varphisupp} and \eqref{varphidecay}  observed earlier. 
This completes the proof of the estimate \eqref{mainbump}.

Now we prove \eqref{leftbump}.
We  write
\[\varphi_{0,k} = \mathbbm{1}_{(-\infty,0)}*\theta_{(2^k)} = 2^k  \widetilde{\theta}_{(2^k)} ,  \]
where   $\widetilde{\theta}:=\mathbbm{1}_{(-\infty,0)}*\theta$ is the primitive of $\theta$. It has high order decay since $\theta$ has integral zero. 
By rescaling, it suffices to show
    \[  \sum_{j=1}^{J} \|M^{\widetilde{\theta}}_{t_j}(f_0,f_1,f_2) - M^{\widetilde{\theta}}_{t_{j-1}}(f_0,f_1,f_2)\|_{2}^2  \lesssim  J^{\frac 12}.\]
This now follows in the same way as \eqref{long3}, using
\[\textup{supp}(\widehat{\widetilde{\theta}}) \subset [-1,-2^{-2}]\cup [2^{-2},1] \]
and high order  decay of $\widetilde{\theta}$.
This completes the proof of \eqref{leftbump}.

It remains to prove \eqref{leftbump1}.
Define \[\varphi_{4,k}(u):= 2^{k} \widetilde{\theta}(u-2^{-k}).\]
We have
\[\varphi_{1,k}(u)=(2^k\widetilde{\theta}(u-2^{-k}))_{(2^k)} =  (\varphi_{4,k})_{(2^k)}(u) .\]
By rescaling, it suffices to show
    \[ 2^{-\gamma k} \sum_{j=1}^{J} \|M^{\varphi_{4,k}}_{t_j}(f_0,f_1,f_2) - M^{\varphi_{4,k}}_{t_{j-1}}(f_0,f_1,f_2)\|_{2}^2  \lesssim  J^{\frac 12}.  \]
 
We split into long and short variation as in \eqref{long3}.
To estimate the short variation, we use the fundamental theorem of calculus and the Cauchy-Schwarz inequality, which bounds
 \begin{equation*}
 \sum_{i=0}^I\,\sum_{j:\, 2^{k_i}< t_j \leq 2^{k_{i}+1}} 2^{-\gamma k}\|M^{\varphi_{4,k}}_{t_{j}}(f_0,f_1,f_2) - M^{\varphi_{4,k}}_{t_{j-1}}(f_0,f_1,f_2)\|_{2}^2   \end{equation*}
\begin{equation}
 \label{shortvarlast}
 \lesssim   \Big[\Big[ \sum_{i=0}^I 2^{-(\gamma+1)k}\int_{\mathbb{R}^3} \int_{1}^2 \big( M_{2^{k_i}t}^{\varphi_{4,k}}(f_0,f_1,f_2)(x) \big)^2\, \frac{dt}{t} dx  \Big]
   \end{equation}
\[ \times \Big[ \sum_{i=0}^I 2^{(1-\gamma)k}\int_{\mathbb{R}^3} \int_{1}^2 \big( M_{2^{k_i}t}^{T\varphi_{4,k}}(f_0,f_1,f_2)(x) \big)^2 \frac{dt}{t} \,dx  \Big]\Big]^{\frac 12}.\]

We are going to estimate each factor in the square brackets as $\lesssim J^{\frac 12}$. 
We begin with the first factor, that we expand as
\begin{equation*} 
\sum_{i=0}^I  \int_{\mathbb{R}^5} \Big[\prod_{n=0}^2 f_n(x+ue_n) 
f_n(x+ve_n) \Big] 
\end{equation*}
\begin{equation*} 
\times\Big[2^{-(\gamma+1)k} \int_1^2  (\varphi_{4,k})_{(2^{k_i}t)} (u) (\varphi_{4,k})_{(2^{k_i}t)} (v)\, \frac{dt}{t} \Big] 
  \,dx du dv. 
\end{equation*}

Similarly as \eqref{longvarcomplsq} and \eqref{ktuv}, this takes the form
\[\Lambda_{D_1,K}((f_s)_{s\in S})\]
with 
\[K(u,v) =  \sum_{i=0}^I  \Big[ \int_1^2  2^{-(\gamma+1)k} (\varphi_{4,k})_{(t)} (u) (\varphi_{4,k})_{(t)} (v)  \,\frac{dt}{t}  \Big]_{(2^{k_i})} =: \sum_{i=0}^I \Phi_{(2^{k_i})}(u,v).\]
We apply Proposition \ref{prop:1} with $\lambda=2-\gamma>1$, using that $\Phi$ is symmetric and positive as a superposition of positive terms, and using 
\[\textup{supp}(\widehat{\Phi}) \subseteq ([-1,-2^{-3}]\cup [2^{-3},1])^2,\] 
\begin{equation}\label{centeredest}|\Phi(u,v)| \le 2^{-(\gamma+1)k}\int_1^2 |\varphi_{4,k}(t^{-1}u)\varphi_{4,k}(t^{-1}v)|\,{dt} 
\end{equation}
\[\leq 2^{(1-\gamma)k}\int_1^2|\widetilde{\theta}(t^{-1}(u-t2^{-k}))\widetilde{\theta}(t^{-1}(v-t2^{-k}))|\,dt\]
\[\lesssim  2^{(1-\gamma)k}\int_1^2
(1+|u+v-t2^{1-k}|)^{-10} (1+|u-v|)^{-10}
\,dt\]
\[\lesssim 2^{(2-\gamma)k}\int_{2^{1-k}}^{2^{2-k}}
(1+|u+v-t|)^{-10} (1+|u-v|)^{-10}
\,dt\]
\[\lesssim 2^{(2-\gamma)k}(1+2^k|u+v|)^{-10} (1+|u-v|)^{-10}.\]
Here we estimated the integral for
$|u+v|<2^{3-k}$ by the integral over $\R$
and for $|u+v|>2^{3-k}$ we estimated the integrand by
its supremum norm.
We used along the way decay estimates of   $\widetilde{\theta}$ 
it inherits from  the window $\chi$.

We turn to the second factor in \eqref{shortvarlast}.
We  proceed as above, in place of \eqref{centeredest} we compute 
\[2^{(1-\gamma)k} \Big|\int_1^2 t\partial_t((\varphi_{4,k})_{(t)}(u))t\partial_{t}((\varphi_{4,k})_{(t)}(v)) \,\frac{dt}{t} \Big|  \]
\[=  2^{(3-\gamma)k}\Big|\int_1^2 t\partial_t(t^{-1}\widetilde{\theta}(t^{-1}(u-t2^{-k})))t\partial_t(t^{-1}\widetilde{\theta}(t^{-1}(v-t2^{-k})))\, \frac{dt}{t} \Big |. \]
Applying Leibniz and chain rules, most terms will be analogous to the above. However,
when a derivative falls on $t2^{-k}$, we obtain a factor $2^{-k}$. The worst term
is the one where both derivatives fall on the $t2^{-k}$. Thus we get the estimate
\[\lesssim 2^{(1-\gamma)k} \int_1^2 (1+|u+v-t2^{1-k}|)^{-10} (1+|u-v|)^{-10}\,\frac{dt}{t}.\]
As above, this is estimated by
\[\lesssim 2^{(2-\gamma)k}(1+2^k|u+v|)^{-10} (1+|u-v|)^{-10}.\]

To treat the long variation, we proceed as for \eqref{longvar-lambda}, where after a bootstrapping estimate we are led to estimate, analogously to \eqref{longvar-lambda-ker}, $\Lambda_{D_1,K}((f_s)_{s\in S})$ with
\[K(u,v) = 2^{-\gamma k}\sum_{i=1}^I(\varphi_{4,k})_{(2^{m_i})}(u)(\varphi_{4,k})_{(2^{m_i})}(v).\]
Similarly as in \eqref{centeredest} we estimate
 \[2^{-\gamma k}|\varphi_{4,k}(u)\varphi_{4,k}(v)| 
 = 2^{(2-\gamma)k} |\widetilde{\theta}(u-2^{-k})\widetilde{\theta}(v-2^{-k})|\]
\[\lesssim  2^{(2-\gamma)k}(1+|u+v-2^{1-k}|)^{-10} (1+|u-v|)^{-10} \]
\[\lesssim 2^{(2-\gamma)k}(1+2^k|u+v|)^{-10} (1+|u-v|)^{-10}.\]
Applying   Proposition \ref{prop:1}
again completes the proof of \eqref{leftbump1}.

\section{Proof of Proposition \ref{prop:1}  using Propositions \ref{prop:2} and
\ref{firstcz}}
\label{sec:prop:1}
Let $\lambda=\frac{3}{2}$. 
Let $k \leq 0$, let $J$ be a positive integer and   $(k_j)_{j=1}^J$  a strictly  increasing sequence of integers.
By splitting into hundred subsequences, using the triangle inequality 
to separate these sequences, we may assume $k_j+100\le k_{j+1}$ for $1\le j<J$.

Let $\Phi_j$ for $1\le j\le J$ be as in Proposition \ref{prop:1}.
In particular, 
$\widehat{\Phi_j}(\xi,-\xi)$ is continuous and even in $\xi$ by the symmetry assumption on the kernel $\Phi_j$.
Furthermore, we claim that 
 $\widehat{\Phi}_j(\xi,-\xi)$ is positive for all $\xi\in \R$.
To see this, first apply Plancherel to the positivity assumption
\eqref{poskernel} in Proposition~\ref{prop:1} to conclude
 \[0\le \int_{\R^2} \widehat{f}(-\xi)\overline{\widehat{f}(\eta)}\widehat{\Phi_j}(\xi,\eta)\, d\xi d\eta\]
 for all Schwartz functions $f$.
 Now we see the claim by using  testing functions $\widehat{f}$ 
 which approximate the Dirac delta at $\xi$.

As $\|\Phi_j\|_1$ has a universal bound, 
for suitable universal constant $c$ we have
$$ \widehat{\Phi_j}(\xi,-\xi)\le c(\widehat{\phi_{0,j}}-\widehat{\phi_{1,j}})(\xi)^2$$
with even real functions $\phi_{0,j}$ and $\phi_{1,j}$,
such that ${(\phi_{0,j})_{2^{-k_j+25}}}$ is a left window and ${(\phi_{1,j})_{2^{-k_j-25}}}$
is a right window.
Moreover, there exists a real even function $\psi_j$ so that
\begin{equation}\label{existsqr}
\widehat{\psi_j}(\xi)^2:=2c (\widehat{\phi_{0,j}}-\widehat{\phi_{1,j}})(\xi)^2-\ \widehat{\Phi_j}(\xi,-\xi).
\end{equation}

Namely, outside the support of $\xi \mapsto \widehat{\Phi_j}(\xi,-\xi)$,
the function $\widehat{\psi_j}$ can be chosen to equal ${\sqrt{2c}}(\widehat{\phi_{0,j}}-\widehat{\phi_{1,j}})$,
while on a neighborhood of this support, the function on the right-hand side is at least $c$ and thus has square root.
The function $\widehat{(\psi_j)_{(2^{-k_j+25})}}$ has support in
$[-1,1]$. 
To understand derivative bounds for this function, let $F(\xi)=\wh{(\Phi_j)_{(2^{-k_j})}}(\xi,-\xi)$. Then we have, for $0\le a\le 8$,
\[
|F^{(a)}(\xi)|=
\Big |(-2\pi i)^a \int_{\R^2} (\Phi_j)_{(2^{-k_j})}(u,v) (u-v)^a e^{-2\pi i \xi  (u - v)} \,du dv \Big |\lesssim 2^{(\lambda-1) k},
\]
by~\eqref{E:phi-lambda}. Thus, 
\begin{equation}\label{psiderivatives}
|(\wh{(\psi_j)_{(2^{-k_j})}})^{(a)}| \lesssim 1,
\end{equation}
as one can see outside the support of $\wh{(\Phi_j)_{(2^{-k_j})}}$
from bounds for derivatives of the windows and on the support using a lower bound
on the right-hand side of \eqref{existsqr} and upper bounds on the derivative of the right-hand side of \eqref{existsqr}.

To show a bound on $\Lambda_{D_1,K}((f_s)_{s\in S})$ with
$K=\sum_{j=1}^J\Phi_j$, which is positive,
it suffices to show a bound on $\Lambda_{D_1,K_0}((f_s)_{s\in S})$
with
$$K_0=\sum_{j=1}^J \Phi_j + \psi_j\otimes {\psi_j}$$
because the form associated with the datum $D_1$ and the difference $K_0-K$ is positive as well.

By Proposition \ref{prop:2}, the form $\Lambda_{D_1,K_1}((f_s)_{s\in S})$
is bounded, where
\[{K_1} ={2c}\sum_{j=1}^J ({\phi_{0,j}}-{\phi_{1,j}})\otimes 
({\phi_{0,j}}-{\phi_{1,j}}).
\]
Hence it suffices to prove a bound on $\Lambda_{D_1,K_3}((f_s)_{s\in S})$, where
$K_3=K_0-K_1$.

This is done by an application of Proposition \ref{firstcz}.
Note that we have on the diagonal
$$\widehat{K_3}(\xi,-\xi)=
\sum_{j=1}^J 
 \widehat{\Phi_j}(\xi,-\xi) + \widehat{\psi_j}(\xi)^2 - {2c}
({\widehat{\phi_{0,j}}}-{\widehat{\phi_{1,j}}})(\xi)^2=0.$$
We verify the remaining assumptions of Proposition \ref{firstcz} for
\[\Psi_j:=\Phi_j + \psi_j\otimes {\psi_j}- {2c}
({\phi_{0,j}}-{\phi_{1,j}})\otimes ({\phi_{0,j}}-{\phi_{1,j}}).\]
We have
\[\textup{supp}(\widehat{\Phi_j}) \subseteq ([-2^{-k_j+20},-2^{-k_j-20}]\cup [2^{-k_j-20},2^{-k_j+20}])^2,\]
\[\textup{supp}((\widehat{\phi_{0,j}}-\widehat{\phi_{1,j}})\otimes (\widehat{\phi_{0,j}}-\widehat{\phi_{1,j}}) ) \subseteq ([-2^{-k_j+25},-2^{-k_j-26}]\cup [2^{-k_j-26},2^{-k_j+25}]{)^2},\]
\[\textup{supp}(\widehat{\psi_j}\otimes \widehat{\psi_j}) \subseteq ([-2^{-k_j+25},-2^{-k_j-26}]\cup [2^{-k_j-26},2^{-k_j+25}])^2.\]
Thus,
\[\textup{supp}(\widehat{\Psi_j}) \subseteq  \{(\xi,\eta)\in \R^2: 2^{-k_j-30} < |(\xi,\eta)| \leq 2^{-k_j+30}\}.\]
Note also that, using in particular \eqref{psiderivatives},
\[|(\Phi_j)_{(2^{-k_j})}(u,v)| \lesssim 2^{\lambda k}  (1+2^k|u+v|)^{-10} (1+|u-v|)^{-10}, \]
\[|((\phi_{0,j}- \phi_{1,j})\otimes (\phi_{0,j}- \phi_{1,j}))_{(2^{-k_j})}(u,v)| \lesssim   (1+|u+v|)^{-4} (1+|u-v|)^{-4}, \]
\[|(\psi_j\otimes \psi_j)_{(2^{-k_j})}(u,v)| \lesssim (1+|u+v|)^{-4} (1+|u-v|)^{-4}. \]
Hence,
\[|(\Psi_j)_{(2^{-k_j})}(u,v)| \lesssim  2^{\lambda k} (1+2^k|u+v|)^{-4} (1+|u-v|)^{-4}+ (1+|u+v|)^{-4} (1+|u-v|)^{-4}. \]
 The final claim now follows from Proposition \ref{firstcz}.

\section{Proof of Proposition \ref{prop:2} using Propositions \ref{firstcz}
and  \ref{firststick}}
 \label{sec:prop:2}

Let $J$ be a positive integer and   $(k_j)_{j=1}^J$ and $(l_j)_{j=1}^J$   two finite sequences of integers with $k_{j}+10< l_j$ for $1\le j\le J$ and
$l_j< k_{j+1}$ for $1\le j<J-1$. By splitting the sequence into 
subsequences of even and odd $j$ if necessary, we may assume without loss of generality that $l_j+10<k_{j+1}$ for each $1\le j<J$.
 Assume a tuple 
 $(f_s)_{s\in S}$ as in \eqref{topbottom} and \eqref{normal1} is given.

Assume we are given $\phi_{0,j}$ and $\phi_{1,j}$ for each $j$ such that ${(\phi_{0,j})_{(2^{-k_j})}}$
is a left window and ${(\phi_{1,j})_{(2^{-l_j})}}$
is a right window.
Pick corresponding functions   
$\psi_{0,j}$ and $\psi_{1,j}$ so that the rescaled functions give universal pairs, and hence 
\begin{equation}\label{firstphipsi}
(1-\widehat{{\phi}_{1,j}})^2+(\widehat{{\psi}_{0,j}})^2=1,
\end{equation}
\begin{equation}\label{secondphipsi}
(\widehat{{\phi}_{0,j}})^2+(1-\widehat{{\psi}_{1,j}})^2=1.
\end{equation}
Then
\begin{equation}\label{diagsumone}
(1-\widehat{\psi_{1,1}})^2+\sum_{j=1}^J (\widehat{\phi_{0,j}}-\widehat{\phi_{1,j}})^2+ \sum_{j=1}^{J-1} (\widehat{\psi_{0,j}}-\widehat{\psi_{1,j+1}})^2+(\widehat{\psi_{0,J}})^2
=1.
\end{equation}
 To see this, note that at every point  at most one of
the functions $\widehat{{\phi}_{0,j}}$, $\widehat{{\phi}_{1,j}}$, $1\le j\le J$
is neither $0$ nor $1$, and the functions $\widehat{{\psi}_{0,j}}$, $\widehat{{\psi}_{1,j}}$
are neither zero nor one precisely when the respective function $\widehat{{\phi}_{1,j}}$,  $\widehat{{\phi}_{0,j}}$ is not zero or one. Therefore, at any point at most one pair
$(\widehat{{\psi}_{0,j}}, \widehat{{\phi}_{1,j}})$ or $(\widehat{{\psi}_{1,j}}, \widehat{{\phi}_{0,j}})$ takes values other than zero and one, and we can apply 
\eqref{firstphipsi} or \eqref{secondphipsi} respectively.

 As $\Lambda_{D_1,K}((f_s)_{s\in S})$ in  Proposition \ref{prop:2} is positive,
 it suffices to estimate its sum 
 with another positive term, and thus it suffices to estimate
 $\Lambda_{D_1,{K_1}}((f_s)_{s\in S})$ with
 \[
\widehat{K_1}(\xi,\eta) =
(1-\widehat{\psi_{1,1}})(\xi) (1-\widehat{\psi_{1,1}})(\eta)+\sum_{j=1}^J (\widehat{\phi_{0,j}}-\widehat{\phi_{1,j}})(\xi)(\widehat{\phi_{0,j}}-\widehat{\phi_{1,j}})(\eta)\]
\[+ \sum_{j=1}^{J-1} (\widehat{\psi_{0,j}}-\widehat{\psi_{1,j+1}})(\xi)
(\widehat{\psi_{0,j}}-\widehat{\psi_{1,j+1}})(\eta)+
(\widehat{\psi_{0,J}})(\xi)(\widehat{\psi_{0,J}})(\eta).\]
This can be rewritten in a more compressed form  
\[\widehat{K_1}=\sum_{j=0}^{2J} (\widehat{\varphi_{0,j}}-\widehat{\varphi_{1,j}})\otimes (\widehat{\varphi_{0,j}}-\widehat{\varphi_{1,j}})\]
where $\widehat{\varphi_{0,0}}=1$, for $1\le j\le J$
\[\varphi_{0,2j-1}=\phi_{0,j},\]
\[\varphi_{0,2j}=\psi_{0,j},\]
\[\varphi_{1,2j-1}=\phi_{1,j},\]
\[\varphi_{1,2j-2}=\psi_{1,j},\]
and ${\varphi_{1,2J}}=0$. 
Define for $1\le j\le J$
\[m_{2j-2}=k_{j}\]
and for $1\le j\le J$
\[m_{2j-1}=l_j.\]
Observe that for each $1\le j\le 2J-1$
we have
$(\varphi_{0,j})_{(2^{-m_{j-1}})}$ is a left window and 
$(\varphi_{1,j})_{(2^{-m_j})}$ is a right window.

In order to apply Proposition \ref{firststick},
we introduce for $0\le j\le 2J$ the functions 
\[\varphi_{2,j}=(\varphi_{1,j})_{(2^{-4})}.\]
Observe that  $(\varphi_{2,j})_{2^{4-m_j}}$ is a  right window whenever $0 \leq j \leq 2J-1$.
We write for $\widehat{K_1}$
\begin{equation}\label{eq:e1}
-\sum_{j=0}^{2J}(\widehat{\varphi_{0,j}} -\widehat{\varphi_{2,j}})\otimes   \widehat{\varphi_{1,j}} 
+\widehat{\varphi_{1,j}}\otimes  (\widehat{\varphi_{0,j}} -\widehat{\varphi_{2,j}})  
\end{equation}
\begin{equation}\label{eq:e2}
-\sum_{j=0}^{2J}(\widehat{\varphi_{2,j}} -\widehat{\varphi_{1,j}})\otimes   \widehat{\varphi_{1,j}}
+\widehat{\varphi_{1,j}}\otimes  (\widehat{\varphi_{2,j}} -\widehat{\varphi_{1,j}})  
\end{equation}
\begin{equation}\label{eq:e3}
+\sum_{j=0}^{2J}\widehat{\varphi_{0,j}}\otimes   \widehat{\varphi_{0,j}}
-\widehat{\varphi_{1,j}}\otimes   \widehat{\varphi_{1,j}}.
\end{equation}

In \eqref{eq:e1}, the bound for the sum of these terms over $1\le j\le 2J-1$
follows from Proposition \ref{firststick}, applied to the
sequence $(m_j)_{j=0}^{2J-1}$ and the rescaled windows
$\varphi_{0,j}$, $\varphi_{1,j}$, $\varphi_{2,j}$
for $1\le j\le 2J-1$.
  The term for $j=2J$ in \eqref{eq:e1} vanishes.
To deal with the term for $j=0$ in \eqref{eq:e1}, we use $\widehat{\varphi_{0,0}}=1$ and rewrite this term as 
\[
-\widehat{\varphi_{1,0}}(\eta)  +\widehat{\varphi_{2,0}}(\xi)   \widehat{\varphi_{1,0}}(\eta) 
+\widehat{\varphi_{1,0}}(\xi) - \widehat{\varphi_{1,0}}(\xi) \widehat{\varphi_{2,0}}(\eta).  \]
Denoting by $f_k$, $k=0,1,2$, the functions defined via~\eqref{E:f} and using the change of variables as in~\eqref{E:change1} and~\eqref{E:change2},
we estimate
\[|\Lambda_{D_1, \varphi_{1,0}\otimes \delta}((f_s)_{s\in S})| = \Big| \int_{\R^4}  \Big[ \prod_{k=0}^2 f_k(x+ue_k)f_k(x) \Big] \varphi_{1,0}(u) \, dx du \Big| \leq \|\varphi_{1,0}\|_1  \lesssim 1\]
where for a fixed $u$ we used H\"older's inequality in $x$ and $\delta$ denotes the Dirac delta at the origin. 
Similarly, 
\[| \Lambda_{D_1, \varphi_{1,0}\otimes \varphi_{2,0}}((f_s)_{s\in S})| \]
\[= \Big| \int_{\R^5} \Big[ \prod_{k=0}^2 f_k(x+ue_k)f_k(x+ve_k) \Big] \varphi_{1,0}(u)\varphi_{2,0}(v) \,dxdudv\Big| \leq  \|\varphi_{1,0}\|_1  \|\varphi_{2,0}\|_1  \lesssim 1.\]
By symmetry, this bounds the form associated with the $j=0$ summand in \eqref{eq:e1}.

It remains to estimate the form associated with $K_2$ where $\widehat{K_2}$ 
is the sum of \eqref{eq:e2}, \eqref{eq:e3}. 
As $\widehat{K_1}$ is constant $1$ on the diagonal $\xi+\eta=0$
by \eqref{diagsumone} and 
the stick terms {\eqref{eq:e1}} vanish on this diagonal, the function  $\widehat{K_2}$
is still constant one on this diagonal.

We define $K_3$ by $\widehat{K_3}:=\widehat{K_2}-1$.
It suffices to prove bounds for the form associated with $K_3$, because
$K_2-K_3$ is the Dirac delta and 
\[|\Lambda_{D_1,K_2-K_3}((f_s)_{s\in S})|=\Big |\int_{\R^3} \prod_{k=0}^{2} f_k^2(x) \, dx \Big |\le 1,\]
where the functions $f_k$ are as in~\eqref{E:f}.
We rewrite $\widehat{K_3}$ as
\begin{equation}\label{eq:e2prime}
-\sum_{j=0}^{2J}(\widehat{\varphi_{2,j}} -\widehat{\varphi_{1,j}})\otimes   \widehat{\varphi_{1,j}}
+\widehat{\varphi_{1,j}}\otimes  (\widehat{\varphi_{2,j}} -\widehat{\varphi_{1,j}})  
\end{equation}
\begin{equation}\label{eq:e3prime}
+\sum_{j=1}^{2J}\widehat{\varphi_{0,j}}\otimes   \widehat{\varphi_{0,j}}
-\widehat{\varphi_{1,j-1}}\otimes   \widehat{\varphi_{1,j-1}},
\end{equation}
where we have reshuffled \eqref{eq:e3} and used $\widehat{\phi_{0,0}}=1$ and $\widehat{\phi_{1,2J}}=0$.
Bounds for the sum of \eqref{eq:e2prime} and \eqref{eq:e3prime} follow from Proposition  \ref{firstcz}. Indeed,  
for each $0\leq j \leq 2J$,  
\[\textup{supp}((\widehat{\varphi_{2,j}} -\widehat{\varphi_{1,j}})\otimes   \widehat{\varphi_{1,j}}) \subseteq ([-2^{-m_j+4},-2^{-m_j-1}]\cup [2^{-m_j-1},2^{-m_j+4}])\times [-2^{-m_j},2^{-m_j}].\]
By symmetry, the $j$-th summand in  \eqref{eq:e2prime} is supported in 
\[ \{(\xi,\eta)\in \R^2: 2^{-m_j-30} < |(\xi,\eta)| \leq 2^{-m_j+30}\} =: A.\]
The  $j$-th summand also satisfies 
 a bound by
 \[ |({\varphi_{2,j}} -{\varphi_{1,j}})\otimes   {\varphi_{1,j}}
+{\varphi_{1,j}}\otimes  ({\varphi_{2,j}} -{\varphi_{1,j}})|_{(2^{-m_j})} (u,v)   \lesssim    (1+|u+v|)^{-4} (1+|u-v|)^{-4} \]
 due to the functions being windows. 

Similarly,   for $0\leq j \leq 2J-1$ we have
 \[\textup{supp}(\widehat{\varphi_{0,j+1}} \otimes \widehat{\varphi_{0,j+1}}- \widehat{\varphi_{1,j}}\otimes   \widehat{\varphi_{1,j}}) \subseteq  [-2^{-m_j},2^{-m_j}]^2\setminus [-2^{-m_j-1},2^{-m_j-1}]^2 \subseteq A \]
and the decay
\[   |{\varphi_{0,j+1}} \otimes {\varphi_{0,j+1}}- {\varphi_{1,j}}\otimes   {\varphi_{1,j}}|_{(2^{-m_j})}(u,v) \lesssim   (1+|u+v|)^{-4} (1+|u-v|)^{-4}. \]
Thus, bounds for $\Lambda_{D_1, K_3}$ follow from  Proposition \ref{firstcz}.

\section{Proof of Proposition \ref{firstcz} using Lemma 3 in \texorpdfstring{\cite{MR4171366}}{}
}
\label{sec:firstcz}
Given a regular $3\times 3$ matrix $A$, let     $D_A$ be the datum defined in \eqref{datumDA}. 
We   recall the following lemma, which is a special instance of a more general result proved in~\cite{MR4171366}. 

\begin{lemma}[\cite{MR4171366}, Lemma 3]\label{L:lemmaDT}
For all $0<\varepsilon<1$, there exists a constant $C$ such that the following holds.

Let $A$ be a regular $3 \times 3$ matrix which differs from $-I$ by at most one row and   satisfies
\begin{equation}\label{E:assumption-A}
|\operatorname{det}A| >\varepsilon \quad \text{and} \quad
\|A\|_{HS} \leq \varepsilon^{-1},
\end{equation}
where $\|A\|_{HS}$ stands for the Hilbert-Schmidt norm of $A$. With $S$ as in the datum $D_A$, let $(f_s)_{s\in S}$ be a tuple of real-valued Schwartz functions such that $\|f_s\|_8=1$ for all $s\in S$. Let $i=1,2,3$ and let $K$ be the kernel satisfying
\begin{equation}\label{E:k-rewritten}
K(\Pi x)=
\int_0^\infty \int_{\R^3} (\partial_i\partial_{i+3} g)_{(t)}(x+((-Ap^T)^T,p))\,dp \frac{dt}{t}.
\end{equation} 
Then
\[
|\Lambda_{D_A,K}((f_s)_{s\in S})| \leq C. 
\]
\end{lemma}

\begin{proof}[Proof of Proposition~\ref{firstcz}]
Let $\lambda=\frac{3}{2}$. 
Let $k\le 0$ be given. Let an integer $J \geq 1$ and a strictly increasing sequence $(k_j)_{j=1}^J$ of integers be given. Let $(\Phi_j)_{j=1}^J$ and $K$ be given as in the proposition.
Let $(f_s)_{s\in S}$ be given as in~\eqref{topbottom} and~\eqref{normal1}.
Set $f_{k}:=f_{(k,0)}=f_{(k,1)}$ for each $k=0,1,2$.

Let $\theta: \R \rightarrow \R$ be a function whose Fourier transform is supported in   $[-2,-1/2] \cup [1/2,2]$ and whose derivatives up to order $8$ are $\lesssim 1$. Assume further that \[\int_0^\infty \wh{\theta}(r\xi)\,\frac{dr}{r}=1\] for all $\xi \neq 0$.
We do the two parameter lacunary decomposition of $\widehat{K}$
in directions $\xi+\eta$ and $\xi-\eta$ and collect these pieces into  lacunary cones away from the line $\xi+\eta=0$ centered at the origin. In detail, we write
\begin{equation}\label{Kintos}
\wh{K}(\xi,\eta)=\int_0^\infty \wh{K^{(z)}} (\xi,\eta)\,\frac{dz}{z}
\end{equation}
with 
\begin{equation}\label{defKs}
\wh{K^{(z)}}(\xi,\eta)=\int_0^\infty 
\wh{K}(\xi,\eta) \wh{\theta}(t(\xi-\eta)) \wh{\theta}(z^{-1}t(\xi+\eta))
\,\frac{dt}{t}.
\end{equation}
We break the integral in \eqref{Kintos} into the integrals over the domains $(0,1)$ and $(1,\infty)$
and do the estimates for these integrals separately.  
We begin with the case $z\in (0,1)$. Here we do an estimate for each $z$ separately
and show for all $z< 1$ that
\begin{equation}\label{E:s-less-1}
|\Lambda_{D_1,K^{(z)}}((f_s)_{s\in S})| \lesssim z^{\frac{(\lambda-1)^2}{2\lambda}} J^{\frac{1}{2}},
\end{equation}
which is an integrable upper bound with respect to the measure $\frac{dz}{z}$. 
Fix $z\in (0,1)$.

Let $g$ be the one-dimensional Gaussian and let $h=g'$. Set $\wh{\omega}=(\wh{h})^{-1}\wh{\theta}$. The function $\wh{\omega}$ satisfies similar support and derivative estimates as $\wh{\theta}$ since $\wh{h}$ and its derivatives are essentially constant on the support of $\wh{\theta}$.
In addition, let $\wh{\phi}$ be a function
supported in the annulus $\frac{1}{16} \leq |(\xi,\eta)| \leq 16$
such that  its derivatives up to order $8$ are $\lesssim 1$ 
and  $\wh{\phi}(\xi,\eta) \wh{g}(\xi) \wh{g}(\eta)=1$ if $1/8 \leq |(\xi,\eta)| \leq 8$.
Then, for all $\xi,\eta\in \R$,
\begin{equation}\label{E:gauss}
\wh{\theta}(\xi-\eta) \wh{\theta}(z^{-1}(\xi+\eta))
=
\wh{\theta}(\xi-\eta)\wh{\omega} (z^{-1}(\xi+\eta))
 \wh{h}(z^{-1}(\xi+\eta))\wh{\phi}(\xi,\eta) \wh{g}(\xi) \wh{g}(\eta). 
\end{equation}
 Note that this equality holds since the left-hand side is supported in the set where 
 \begin{equation}\label{E:support}
 1/8 \leq  |(\xi,\eta)| \leq 8.
 \end{equation} 
 Indeed, on the support of the left-hand side of~\eqref{E:gauss} we have
 $|\xi+\eta| \leq 2z \leq 2$ and $1/2 \leq |\xi-\eta| \leq 2$.
This yields~\eqref{E:support}. 
 
For $z\in (0,1)$ and $t>0$ we define the function $w^{z,t}$ via
\begin{equation}\label{E:def-wst}
\wh{w^{z,t}}(\xi,\eta)= \wh{K}(t^{-1}(\xi,\eta))\wh{\phi}(\xi,\eta)\wh{\theta}(\xi-\eta) \wh{\omega} (z^{-1}(\xi+\eta)).
\end{equation}
Let $\Pi$ be the projection associated with the datum $D_1$. Using the Fourier inversion formula and equations \eqref{defKs}, \eqref{E:gauss} and~\eqref{E:def-wst}, we write ${K^{(z)}}(\Pi x)$ as
\begin{equation}\label{E:K_Pi}
\int_0^\infty \int_{\R^2}   \wh{w^{z,t}}(t(\xi,\eta))
\wh{h}(z^{-1}t(\xi+\eta))
\wh{g}(t\xi) \wh{g}(t\eta) 
e^{2\pi i( \xi (x_3^0-x_0-x_1-x_2)+\eta(x_3^1-x_0-x_1-x_2))}\,d\xi d\eta \frac{dt}{t}.
\end{equation}
Since the Fourier transform of $w^{z,t}$ is supported in the set where $\frac{1}{8} \leq |(\xi,\eta)| \leq 8$, we observe that $w^{z,t}$ vanishes unless $t$ is in the set
$$ M:= \bigcup_{j=1}^J [2^{k_j-33} , 2^{k_j+33}].$$
We may thus restrict the region of  $t$-integration in~\eqref{E:K_Pi} to $M$.
Further, we may interpret the inner integral in~\eqref{E:K_Pi} as the integral of the Fourier transform of the function 
\begin{equation*}
(y_0,y_1,y_2,y_3,y_4) \mapsto
\end{equation*}
\begin{equation*}w^{z,t}_{(t)}(y_0+x_0+x_1,y_1+x_0+x_1)h_{(z^{-1}t)}(y_2+x_2)
g_{(t)}(y_3+x_3^0) 
g_{(t)}(y_4+x_3^1)   
\end{equation*}
over the hyperplane
$$
\{(-\xi,-\eta,-\xi-\eta,\xi,\eta):~\xi \in \R, ~\eta \in\R\}.
$$
It is therefore up to universal multiplicative constant equal to the integral of the function itself over the orthogonal complement
$$
\{(p+q-r,q-r,r,p+q,q):~p, q, r\in \R\}.
$$
The form $\Lambda_{D_1,K^{(z)}}((f_s)_{s\in S})$ can then be rewritten as
\begin{equation}\label{E:form}
\int_M \int_{\R^8}  \Big[\prod_{s\in S} f_s(\Pi_s x)\Big] 
w^{z,t}_{(t)}(x_0+x_1+p+q-r, x_0+x_1+q-r) 
\end{equation}
\begin{equation*}
\times h_{(z^{-1}t)}(x_2+r) g_{(t)}(x_3^0+p+q)
g_{(t)}(x_3^1+q) 
\,dx dp dq dr \frac{dt}{t}.
\end{equation*}
 We write the integral in $x_2$ as the innermost and use the Cauchy-Schwarz inequality in the remaining variables. This bounds~\eqref{E:form} by the geometric mean of
\begin{equation}\label{E:second-term}
\int_{M} \int_{\R^7} \Big[\prod_{i=0,1} |f_{2}(x_0,x_1,x_3^i)|^2 \Big ]  
|w^{z,t}_{(t)}(x_0+x_1+p+q-r, x_0+x_1+q-r)| 
\end{equation}
\begin{equation*}
\times g_{(t)}(x_3^0+p+q)
g_{(t)}(x_3^1+q) 
 \,dx_0 dx_1 dx_3^0 dx_3^1 dp dq dr \frac{dt}{t}
\end{equation*}
and
\begin{equation*}
\int_{0}^\infty \int_{\R^7} 
\Big[
\int_{\R} \Big [ \prod_{i=0,1}f_0(x_3^i,x_1,x_2)f_1(x_0,x_3^i,x_2) \Big ] 
h_{(z^{-1}t)}(x_2+r) \,dx_2\Big]^2 \end{equation*}
\begin{equation*}
\times |w^{z,t}_{(t)}(x_0+x_1+p+q-r, x_0+x_1+q-r)| 
\end{equation*}
\begin{equation}\label{E:first-term}
\times g_{(t)}(x_3^0+p+q)
g_{(t)}(x_3^1+q) 
\,dx_0 dx_1 dx_3^0 dx_3^1 dp dq dr \frac{dt}{t}.
\end{equation}

In order to bound~\eqref{E:second-term} and~\eqref{E:first-term}, we prove a pointwise estimate for $w^{z,t}$. We first claim
\begin{equation}\label{E:first-estimate}
|w^{z,t}(u,v)| 
\lesssim z^\lambda.
\end{equation}
To verify the claim, we observe that
since $\wh{K}$ vanishes on the diagonal $\xi+\eta=0$, the function $\wh{K_{(t^{-1})} \ast \phi}$ has the same property. 
Therefore
\begin{equation}\label{E:der-est}
    |\wh{K_{(t^{-1})} \ast \phi}(\xi,\eta)|
=|\wh{K_{(t^{-1})} \ast \phi}(\xi,\eta)-\wh{K_{(t^{-1})} \ast \phi}((\xi-\eta)/2,-(\xi-\eta)/2)|
\end{equation}
\begin{equation*}
=\Big|\int_{\R^2} K_{(t^{-1})} \ast \phi(u,v) e^{-\pi i(\xi-\eta)(u-v)} (e^{-\pi i(\xi+\eta)(u+v)}-1) \,du dv\Big|
\end{equation*}
\begin{equation*}
\lesssim \int_{\R^2} |K_{(t^{-1})} \ast \phi(u,v)| \min\{|\xi+\eta||u+v|,1\}\,du dv
\end{equation*}
\begin{equation*}
\leq |\xi+\eta|^{\lambda-1} \int_{\R^2} |K_{(t^{-1})} \ast \phi(u,v)| ~ |u+v|^{\lambda-1}\,du dv,
\end{equation*}
as $\lambda-1 \in (0,1)$.
We observe that 
\begin{equation}\label{E:kt}
|K_{(t^{-1})} \ast \phi(u,v)| \lesssim 2^{\lambda k}(1+2^k|u+v|)^{-4} (1+|u-v|)^{-4} +  (1+|u+v|)^{-4} (1+|u-v|)^{-4},
\end{equation}
thanks to the derivative estimates on $\phi$, to the support properties of $\wh{\phi}$ and $\wh{\Phi_j}$ and to~\eqref{E:phi-lambda3/2}.
Therefore,
\begin{equation*}
\int_{\R^2} |K_{(t^{-1})} \ast \phi(u,v)|~ |u+v|^{\lambda-1}\,du dv
\end{equation*}
\begin{equation*}
\lesssim 2^{k} \int_{\R^2} (1+2^k|u+v|)^{\lambda-5} (1+|u-v|)^{-4}\, du dv
+ \int_{\R^2} (1+|u+v|)^{\lambda-5} (1+|u-v|)^{-4}\, du dv
\lesssim 1.
\end{equation*}
 Combining this with~\eqref{E:der-est} and passing to $w^{z,t}$, we thus obtain
\[
|\wh{w^{z,t}}(\xi,\eta)|
\lesssim z^{\lambda-1}. 
\]
Estimating the Fourier inversion formula by $L^1 \rightarrow L^\infty$ bounds, inequality~\eqref{E:first-estimate} follows.

We note that the right-hand side of~\eqref{E:first-estimate} has the desired decay as $z$ tends to $0$, however, it does not have a good behavior with respect to $(u,v)$. We therefore derive a yet another estimate for $w^{z,t}$ in which the right-hand side possesses merely $L^1$ scaling in $z$ but decays sufficiently fast as $|(u,v)|$ tends to infinity. We set 
\[
F(u,v)=\omega_{(z^{-1})}((u+v)/2) \theta((u-v)/2).
\]
By~\eqref{E:def-wst}, we have
$w^{z,t}=K_{(t^{-1})} \ast \phi \ast F$. 
Recall that the functions $\wh{\omega}$ and $\wh{\theta}$ are supported in $[-2,2]$ and have derivatives up to order $8$ bounded by $\lesssim 1$. 
Using~\eqref{E:kt}, we therefore obtain
\begin{equation}\label{E:I}
|w^{z,t}(u,v)| \lesssim 2^{\lambda k}(1+2^k|u+v|)^{-4} (1+|u-v|)^{-4}
+z(1+z|u+v|)^{-4} (1+|u-v|)^{-4} \quad \text{if } 2^k \leq z
\end{equation}
and
\begin{equation}\label{E:II}
|w^{z,t}(u,v)| \lesssim z(1+z|u+v|)^{-4} (1+|u-v|)^{-4} \quad \text{if } z \leq 2^k.
\end{equation}

Finally, we write $|w^{z,t}|=|w^{z,t}|^{\frac{\lambda-1}{2\lambda}} |w^{z,t}|^{\frac{\lambda+1}{2\lambda}}$ and use the estimate~\eqref{E:first-estimate} for the first factor and the estimates~\eqref{E:I} and~\eqref{E:II} for the second factor. This yields the desired bounds
\begin{equation}\label{E:wzt-pointwise}
|w^{z,t}(u,v)| \lesssim z^{\frac{(\lambda-1)^2}{2\lambda}} [z(1+z|u+v|)^{-2-\frac{2}{\lambda}}+2^k(1+2^k|u+v|)^{-2-\frac{2}{\lambda}}] (1+|u-v|)^{-2-\frac{2}{\lambda}}  
\end{equation}
if  $2^k \leq z$, and
\begin{equation}\label{E:wzt-pointwise-2}
|w^{z,t}(u,v)| \lesssim z^{\frac{(\lambda-1)^2}{2\lambda}} z(1+z|u+v|)^{-2-\frac{2}{\lambda}} (1+|u-v|)^{-2-\frac{2}{\lambda}} \quad \text{if } z \leq 2^k.
\end{equation}

Having inequalities~\eqref{E:wzt-pointwise} and~\eqref{E:wzt-pointwise-2} at our disposal, we proceed to bound the term~\eqref{E:second-term}.
We observe that this term can be written as
\begin{equation}\label{E:second-rewritten}
\int_{M} \int_{\R^5}  \Big[\prod_{i=0,1} |f_{2}(x_0,x_1,x_3^i)|^2 \Big ]  
\end{equation}
\begin{equation*}
\times [|w^{z,t}| \ast (g \otimes g)]_{(t)}(x_3^0-x_0-x_1+r, x_3^1-x_0-x_1+r) 
\, dx_0 dx_1 dx_3^0 dx_3^1 dr \frac{dt}{t}.
\end{equation*}
Applying the  Cauchy-Schwarz inequality, we bound \eqref{E:second-rewritten}  with \[v_{z,t}:=[|w^{z,t}| \ast (g \otimes g)]_{(t)}\] by
\begin{equation*}
    \int_M \prod_{i=0,1}\Big[\int_{\R^5}|f_2(x_0,x_1,x_3^i)|^4     
v_{z,t}(x_3^0-x_0-x_1+r,x_3^1-x_0-x_1+r) \,  dx_0 dx_1 dx_3^0 dx_3^1 dr \Big]^{\frac 1 2}   \,\frac{dt}{t}.
\end{equation*}
The product of the square roots of the integrals for $i=0,1$ equals
\begin{equation*}
    \|f_2\|_4^4 \|v_{z,t}\|_1\lesssim z^{\frac{(\lambda-1)^2}{2\lambda}}.
\end{equation*}
The last identity can be seen by integrating first 
in $x_3^{1-i}$ and then in $r$ to obtain the $L^1$ norm of $v_{z,t}$. What remains is then the $L^4$ norm of $f_2$ raised to the fourth power.
Using that $\int_M \frac{dt}{t} \lesssim J$, we deduce that~\eqref{E:second-term} is bounded by a multiple of \[z^{\frac{(\lambda-1)^2}{2\lambda}}J.\]

We next focus on the term~\eqref{E:first-term}.
Using the estimates~\eqref{E:wzt-pointwise} and~\eqref{E:wzt-pointwise-2}, bounding the form~\eqref{E:first-term} reduces to estimating
\begin{equation*}
\int_{0}^\infty \int_{\R^7} 
\Big[
\int_{\R} \Big[  \prod_{i=0,1} f_0(x_3^i,x_1,x_2) f_1(x_0,x_3^i,x_2) \Big] h_{(z^{-1}t)}(x_2+r) \,dx_2 \Big]^2 
\end{equation*}
\begin{equation*}
\times t^{-1} \gamma (1+t^{-1}\gamma|x_0+x_1+p/2+q-r|)^{-2-\frac{2}{\lambda}}
t^{-1}(1+t^{-1}|p|)^{-2-\frac{2}{\lambda}}  
\end{equation*}
\begin{equation*}
\times g_{(t)}(x_3^0+p+q)
g_{(t)}(x_3^1+q) \,
dx_0 dx_1 dx_3^0 dx_3^1 dp dq dr \frac{dt}{t},
\end{equation*}
where $\gamma=z$, or $\gamma=2^{k}$ if $2^k \leq z$. We will prove a bound independent of $z$ and $k$, which will bound~\eqref{E:first-term} by $\lesssim z^{\frac{(\lambda-1)^2}{2\lambda}}$ thanks to the extra factor $z^{\frac{(\lambda-1)^2}{2\lambda}}$ in~\eqref{E:wzt-pointwise} and~\eqref{E:wzt-pointwise-2}.

We dominate
\begin{equation*}
t^{-1} \gamma (1+t^{-1}\gamma|x_0+x_1+p/2+q-r|)^{-2-\frac{2}{\lambda}}
t^{-1}(1+t^{-1}|p|)^{-2-\frac{2}{\lambda}}
\end{equation*}
\begin{equation*}
\lesssim t^{-2}\gamma (1+t^{-1}|(\gamma(x_0+x_1+p/2+q-r), 2p)|)^{-2-\frac{2}{\lambda}}
\end{equation*}
\begin{equation*}
\lesssim \int_2^\infty g_{(\alpha \gamma^{-1}t)}(x_0+x_1+p/2+q-r) g_{(\alpha t)}(2p) \,\frac{d\alpha}{\alpha^{1+\frac{2}{\lambda}}}.
\end{equation*}

It thus suffices to estimate the form
\begin{equation}
\label{E:first-term2}
\int_{0}^\infty \int_{\R^7} 
\Big[
\int_{\R} 
\Big[  \prod_{i=0,1} f_0(x_3^i,x_1,x_2) f_1(x_0,x_3^i,x_2) \Big] h_{(z^{-1}t)}(x_2+r) \,dx_2 \Big]^2 
\end{equation}
\begin{equation*}
\times g_{(\alpha \gamma^{-1}t)}(x_0+x_1+p/2+q-r) g_{(\alpha t)}(2p)
g_{(t)}(x_3^0+p+q)
g_{(t)}(x_3^1+q) \,
dx_0 dx_1 dx_3^0 dx_3^1 dp dq dr\frac{dt}{t}
\end{equation*}
with a bound $\lesssim \alpha$ and then integrate over $\alpha$, using that $2/\lambda>1$.
We claim that
\begin{equation}\label{E:convolution}
\int_{\R} g_{(\alpha \gamma^{-1}t)}(x_0+x_1+p/2+q-r) g_{(\alpha t)}(2p)
g_{(t)}(x_3^0+p+q) \,dp
\end{equation}
\begin{equation*}
\lesssim g_{(2^{1/2}\alpha \gamma^{-1}t)}(x_0+x_1+q-r) g_{(\alpha t)}(x_3^0+q). 
\end{equation*}
Indeed, we have
\[
g_{(\alpha t)}(2p)
\lesssim e^{-\pi \gamma^2 \alpha^{-2} t^{-2} (p/2)^2} g_{(2^{-1/2}\alpha t)}(p).
\]
The elementary inequality $e^{-2(a+b)^2} e^{-2b^2} \leq e^{-a^2}$ yields
\[
g_{(\alpha \gamma^{-1}t)}(x_0+x_1+p/2+q-r) e^{-\pi \gamma^2 \alpha^{-2} t^{-2} (p/2)^2}
\lesssim g_{(2^{1/2}\alpha \gamma^{-1}t)}(x_0+x_1+q-r).
\]
Thus, the left-hand side of~\eqref{E:convolution} is bounded by
\[
g_{(2^{1/2}\alpha \gamma^{-1}t)}(x_0+x_1+q-r) (g_{(2^{-1/2}\alpha t)} \ast g_{(t)})(x_3^0+q)
\lesssim g_{(2^{1/2}\alpha \gamma^{-1}t)}(x_0+x_1+q-r) g_{(\alpha t)}(x_3^0+q),
\]
as desired.

Expressing further $g_{(2^{1/2}\alpha \gamma^{-1}t)}(x_0+x_1+q-r)$ as a convolution of two Gaussians and using the evenness of the Gaussian, \eqref{E:first-term2} is bounded by
\begin{equation}\label{E:first-term4}
\alpha \int_{0}^\infty \int_{\R^7} 
\Big[
\int_{\R} 
\Big[  \prod_{i=0,1} f_0(x_3^i,x_1,x_2) f_1(x_0,x_3^i,x_2) \Big] h_{(z^{-1}t)}(x_2+r) \,dx_2 \Big]^2  
\end{equation}
\begin{equation*}
\times g_{(\alpha \gamma^{-1}t)}(x_0+p) g_{(\alpha \gamma^{-1}t)}(x_1-p+q-r)  
g_{(\alpha t)}(x_3^0+q)
g_{(\alpha t)}(x_3^1+q) \,
dx_0 dx_1 dx_3^0 dx_3^1 dp dq dr \frac{dt}{t}.
\end{equation*}

After renaming of variables, naming the variable $x_2$ that is twice an integration variable once as $x_2^0$ and once as $x_2^1$, then 
renaming the variables $x_0,x_1,x_2^0,x_2^1,x_3^0,x_3^1$ in this order as
$x_1^1,x_1^0,x_3^0,x_3^1,x_2^0,x_2^1$, and finally 
introducing functions $\widetilde{f}_0(a,b,c)=f_0(b,a,c)$ and $\widetilde{f}_1=f_1$, we write~\eqref{E:first-term4} as
\begin{equation*}
\alpha \int_{0}^\infty \int_{\R^7} 
\Big[ \prod_{i=0,1}\int_{\R}  \widetilde{f}_0(x_1^0, x_2^{0},x_3^i) \widetilde{f}_1(x_1^1,x_2^0,x_3^i) \widetilde{f}_0(x_1^0,x_2^1,x_3^i) \widetilde{f}_1(x_1^1,x_2^1,x_3^i)  h_{(z^{-1}t)}(x_3^i+r) \,dx_3^i \Big ]   \end{equation*}
\begin{equation*}
\times g_{(\alpha \gamma^{-1}t)}(x_1^0-p+q-r) g_{(\alpha \gamma^{-1}t)}(x_1^1+p) 
g_{(\alpha t)}(x_2^0+q)
g_{(\alpha t)}(x_2^1+q)\,
dx_1^0 dx_1^1 dx_2^0 dx_2^1 dp dq dr \frac{dt}{t}.
\end{equation*}
Let $S$ and $(\Pi_s)_{s\in S}$ be as in the datum $D_A$.
Introducing 
$f_{s}=\widetilde{f}_{s(1)}$ for $s \in S$,  we may write the last display as
\begin{equation*}
\alpha \int_{0}^\infty \int_{\R^9} \Big[ \prod_{s\in S} f_{s}(\Pi_{s} x) \Big] 
g_{(\alpha \gamma^{-1}t)}(x_1^0-p+q-r) g_{(\alpha \gamma^{-1}t)}(x_1^1+p)  
\end{equation*}
\begin{equation*}
\times g_{(\alpha t)}(x_2^0+q)
g_{(\alpha t)}(x_2^1+q) h_{(z^{-1}t)}(x_3^0+r) h_{(z^{-1}t)}(x_3^1+r)\,
dx dp dq dr \frac{dt}{t}.
\end{equation*}

Let $V: \R^3 \rightarrow \R^3$ be a mapping given by $V(v_0,v_1,v_2) = (\alpha \gamma^{-1}v_0, \alpha v_1,z^{-1}v_2)$. We perform the change of variables with respect to this mapping for each of the triples $(p,q,r)$, $(x_1^{0},x_2^{0},x_3^{0})$ and $(x_1^{1},x_2^{1},x_3^{1})$. After this transformation, the above form becomes
\begin{equation}\label{E:first-term5}
\alpha \int_{0}^\infty \int_{\R^9} \Big[\prod_{s\in S} \alpha^{\frac{1}{4}} \gamma^{-\frac{1}{8}} z^{-\frac{1}{8}} f_{s}(V \Pi_s x) \Big] 
g_{(t)}(x_1^0-p+\gamma q-\gamma z^{-1} \alpha^{-1}r) g_{(t)}(x_1^1+p) 
\end{equation}
\begin{equation*}
\times g_{(t)}(x_2^0+q)
g_{(t)}(x_2^1+q) h_{(t)}(x_3^0+r) h_{(t)}(x_3^1+r)\,
dx dp dq dr \frac{dt}{t}.
\end{equation*}
This can be recognized as $\alpha$ multiple of
\[
\Lambda_{D_A,K}(( \alpha^{\frac{1}{4}} \gamma^{-\frac{1}{8}} z^{-\frac{1}{8}} f_{s} \circ V)_{s\in S}),
\]
where $K$ has the form~\eqref{E:k-rewritten} with $i=3$ and 
$$
A=
\left ( \begin{array}{ccc}
1 & -\gamma & \gamma z^{-1} \alpha^{-1}\\
0 & -1 & 0 \\
0 & 0 & -1
\end{array} \right ).
$$
Since $0<\gamma \leq z \leq 1 \leq \alpha$, the matrix $A$ satisfies the assumption~\eqref{E:assumption-A} with $\varepsilon=5^{-1/2}$. 
 Observing further that the function $\alpha^{\frac{1}{4}} \gamma^{-\frac{1}{8}} z^{-\frac{1}{8}} f_{s} \circ V$ has the same $L^8$-norm as $f_{s}$, we deduce from Lemma~\ref{L:lemmaDT} that~\eqref{E:first-term5} is bounded by $\lesssim \alpha$.
This yields the desired bound for~\eqref{E:first-term}.

Combining the estimates for~\eqref{E:second-term} and~\eqref{E:first-term}, we obtain~\eqref{E:s-less-1}.

It remains to consider the part of the integral in~\eqref{Kintos} where $z\in (1,\infty)$. Let $\varphi$ be the function defined via its Fourier transform by \[\wh{\varphi}(\xi)=\int_1^\infty  \wh{\theta}(z\xi) \,\frac{dz}{z}.\] Then we can write
\begin{equation}\label{E:s=1}
\int_1^\infty \wh{K^{(z)}} (\xi,\eta)\,\frac{dz}{z}
=\int_0^\infty 
\wh{K}(\xi,\eta) \wh{\varphi}(t(\xi-\eta)) \wh{\theta}(t(\xi+\eta))\,
\frac{dt}{t}.
\end{equation}
Formally, this expression has the same form as~\eqref{defKs} when $z=1$, except that the function $\wh{\theta}$ is at one occurrence replaced by $\wh{\varphi}$. 
Due to this similarity, we will denote~\eqref{E:s=1} by $\wh{K^{(1)}}(\xi,\eta)$. Note that 
$\wh{\theta}$ is supported in $[-2,-1/2] \cup [1/2,2]$, $\wh{\varphi}$ is supported in $[-2,2]$ and the support properties of $\wh{\theta}$ and $\wh{\varphi}$ ensure that $\wh{\varphi}(\xi-\eta) \wh{\theta}(\xi+\eta)$ is supported in the set where $1/8 \leq |(\xi,\eta)| \leq 8$. We may therefore apply an argument analogous to the case $z\in (0,1)$, arriving at the estimate
\[
|\Lambda_{D_1,K^{(1)}}((f_s)_{s\in S})| \lesssim J^{\frac{1}{2}}.
\]
Combining this with~\eqref{E:s-less-1} yields the conclusion of the proposition.
\end{proof}

\section{Proof of Proposition \ref{firststick} using Propositions \ref{thm:4} and ~\ref{secstickphi}}
\label{sec:firststick}

Let $(k_j)_{j=0}^J$ be a finite increasing sequence of integers with $k_{j-1}+10 \leq k_{j}$. For $1\le j\le J$, let $\phi_{0,j}$, $\phi_{1,j},\phi_{2,j}$
be rescaled respective left or right windows as in the proposition,
and define $K$ as in \eqref{kerfirststick}.
Let  $(f_s)_{s\in S}$ be a tuple of functions as in
\eqref{topbottom} and \eqref{normal1}. 
Set $f_{k}:=f_{(k,0)}=f_{(k,1)}$ for each $k=0,1,2$.

Let  {$({\chi},\phi)$} be  a universal pair and define
$\chi_j:=\chi_{(2^{k_j-2})}$ and $\phi_j:=\phi_{(2^{k_j-2})}$. 
Define 
\begin{equation}
    \label{phi3}
    {\phi_{3,j}}:={\chi_{j-1}}-{\phi_j}
\end{equation}  
and consequently, 
\begin{equation*}
    (\widehat{\phi_{3,j}})^2 =(\widehat{\chi_{{j-1}}})^2- (\widehat{\chi_{{j}}})^2 .
\end{equation*}

If $(\xi,\eta)$ is in the support of $(\widehat{\phi_{0,j}} -\widehat{\phi_{2,j}})\otimes \widehat{\phi_{1,j}}$, then 
\[ 2^{-k_j+3}-2^{-k_j}\leq |\xi+\eta|\leq 2^{-k_{j-1}}+2^{-k_j}. \]
In this range, $\widehat{\chi_{{j-1}}}(\xi+\eta)$ is constant $1$ and
$\widehat{\chi_{{j}}}(\xi+\eta)$ is constant zero. 
We can therefore introduce  artificial factors $\widehat{\phi_{3,j}}$ in $\widehat{K}$ as follows,
\begin{equation*}
 \widehat{K}(\xi,\eta) = \sum_{j=1}^J (\widehat{\phi_{0,j}} -\widehat{\phi_{2,j}})(\xi)  \widehat{\phi_{1,j}}(\eta)  =  \sum_{j=1}^J (\widehat{\phi_{0,j}} -\widehat{\phi_{2,j}})(\xi)  \widehat{\phi_{1,j}}(\eta)\widehat{\phi_{3,j}}(-\xi-\eta). 
\end{equation*} 
Taking the Fourier transform, we obtain for some universal constant $C$,
$$K(u,v)= \int_{\R^2}\sum_{j=1}^J (\widehat{\phi_{0,j}} -\widehat{\phi_{2,j}})(\xi)e^{2\pi i u\xi}  \widehat{\phi_{1,j}}(\eta)e^{2\pi i v\eta}\widehat{\phi_{3,j}}(-\xi-\eta)\, d\xi d\eta$$
\begin{equation}\label{takeft1}=C  \sum_{j=1}^J \int_\R({\phi_{0,j}} -{\phi_{2,j}})(u+p)  {\phi_{1,j}}(v+p){\phi_{3,j}}(p) \,dp,
\end{equation}
where we used that the integral of a function in $\R^3$ over the diagonal
$\{(p,p,p): p\in \R\}$ equals the integral of its Fourier transform over the orthogonal complement of the diagonal, suitably normalized. 

Therefore, with $S$ and $\Pi_s$ as in the datum $D_1$, and doing a variable transformation $p\to x_2+p$,
\[\Lambda_{D_1,K}((f_s)_{s\in S}) =   \sum_{j=1}^J\int_{{\R^6}}  \Big[ \prod_{s\in S}f_s(\Pi_s x) \Big]  \] 
  \begin{equation}
      \label{formtobd}
  \times  ({\phi_{0,j}} -{\phi_{2,j}})(x_3^0-x_0-x_1+p)  {\phi_{1,j}}(x_3^1-x_0-x_1+p)\phi_{3,j} (x_2+p) \,dx dp. 
  \end{equation}

We  apply Fubini  in \eqref{formtobd} to have the integral in $x_2$ as the innermost and then apply  the  Cauchy-Schwarz inequality in $x_0,x_1,x_3^0,x_3^1, p$, which bounds $|\Lambda_{D_1,K}((f_s)_{s\in S})|$ 
up to a constant by  the geometric mean of 
   \begin{equation}
 \sum_{j=1}^J  \int_{\R^5}  \Big[\prod_{i=0,1} |f_2(x_0,x_1,x_3^i)|^2 \Big ]   \mu_j(x_3^0-x_0-x_1+p,x_3^1-x_0-x_1+p) \,  dx_0 dx_1 dx_3^0 dx_3^1 dp
      \label{trv}
 \end{equation}
 and
\[
 \sum_{j=1}^J \int_{\R^5}  \Big[ \int_{\R} \Big[  \prod_{i=0,1} f_0(x_3^i, x_1,x_2) f_1(x_0,x_3^i,x_2)  \Big]
  \phi_{3,j} (x_2+p)\, dx_2 \Big]^2       \]
 \begin{equation}
     \label{positive}
    \times \mu_j(x_3^0-x_0-x_1+p,x_3^1-x_0-x_1+p) dx_0 dx_1 dx_3^0 dx_3^1    dp ,
 \end{equation}
where we have introduced the weight $\mu_j$ defined by
$$\mu_j(u,v)=| {\phi_{0,j}} -{\phi_{2,j}}|(u)  |\phi_{1,j}|(v).$$
We will estimate  \eqref{trv} as $\lesssim J$ and \eqref{positive} as $\lesssim 1$, thereby proving
Proposition \ref{firststick}.

We begin with  \eqref{trv}. Applying the  Cauchy-Schwarz inequality in the remaining integration variables,  we bound \eqref{trv}  by
\begin{equation*}
    \sum_{j=1}^J  \prod_{i=0,1} \Big[ \int_{\R^5}|f_2(x_0,x_1,x_3^i)|^4     
\mu_j(x_3^0-x_0-x_1+p,x_3^1-x_0-x_1+p)  \, dx_0 dx_1 dx_3^0 dx_3^1 dp \Big]^{\frac 12}   
\end{equation*}
\[= \sum_{j=1}^{J}\|f_2\|_4^4 \|\mu_j\|_1\lesssim J.\]
Here the identity is seen by integrating first 
in $x_3^{1-i}$ then in $p$ to obtain the $L^1$ norm of $\mu_j$.

It remains to estimate \eqref{positive}.
We use decay of $\mu_j$ thanks to control of derivatives of Fourier transform of windows and the superposition estimate
\begin{equation*}
    (1+|(u,v)|)^{-N-20} \lesssim \int_1^\infty g_{(\alpha)}(u)
    g_{(\alpha)}(v)
    \, \frac{d\alpha}{\alpha^{N+10}},
\end{equation*}
which we scale isotropically and unisotropically, to dominate
$$\mu_j(u,v)\lesssim
\int_{1}^\infty g_{(\alpha 2^{k_j})}(u) g_{(\alpha 2^{k_j})}(v) 
\, \frac{ d\alpha}{\alpha^{N+10}}
 +\int_1^\infty g_{(\alpha 2^{k_{j-1}})}(u) g_{(\alpha 2^{k_{j}})}(v) 
\,  \frac{ d\alpha}{\alpha^{N+10}}.
 $$
By superposition of positive terms, it suffices to estimate as
$\lesssim \alpha^{N}$ the variant of \eqref{positive} with  $\mu_j(u,v)$ replaced by
$$g_{(\alpha2^{l_j})}(u)g_{(\alpha2^{k_{j}})}(v)$$
and for each of the sequences $l_j=k_j$ and $l_j=k_{j-1}$. 
Define the sequence of real numbers $(m_j)_{j=1}^J$ by
$$2^{2m_j}+2^{2m_j}=2^{2k_j}+ 2^{2l_j}.$$
Note that $k_j-1\le m_j\le k_j$, because $l_j\le k_j$.
Adding and subtracting terms, it suffices to estimate as
$\lesssim \alpha^{N}$ the variants of \eqref{positive} 
with  $\mu_j(u,v)$ replaced by
\begin{equation}\label{symmw}
    g_{(\alpha 2^{m_j})}(u)g_{(\alpha 2^{m_j})}(v)
\end{equation}
and by $(\nu_j)_{(\alpha)}(u,v)$, where
\begin{equation}\label{diffw}
\nu_j(u,v):= 
g_{( 2^{l_j})}(u)g_{( 2^{k_j})}(v)-
 g_{( 2^{m_j})}(u)g_{( 2^{m_j})}(v).
\end{equation}

We begin with  \eqref{symmw}. We need to estimate
\[
 \sum_{j=1}^J \int_{\R^5}  \Big[ \int_{\R} \Big[ \prod_{i=0,1}f_0(x_3^i, x_1,x_2) f_1(x_0,x_3^i,x_2)\Big ]   \phi_{3,j} (x_2+p) \,dx_2 \Big]^2       \]
 \begin{equation}
     \label{positive2}
 \times g_{(\alpha 2^{m_j})}(x_3^0-x_0-x_1+p)g_{(\alpha 2^{m_j})}(x_3^1-x_0-x_1+p) \,  dx_0 dx_1 dx_3^0 dx_3^1   dp .
 \end{equation}
A renaming of variables, naming the variable $x_2$ that is twice an integration variable once as $x_2^0$ and once as $x_2^1$, then 
renaming the variables $x_0,x_1,x_2^0,x_2^1,x_3^0,x_3^1$ in this order as
$x_1,x_0,x_3^0,x_3^1,x_2^0,x_2^1$, and finally 
introducing functions $\widetilde{f}_0(a,b,c)=f_0(b,a,c)$ and $\widetilde{f}_1=f_1$, we write \eqref{positive2} as
\[
 \sum_{j=1}^J \int_{\R^5}  \Big[ \prod_{i=0}^1\int_{\R}  \widetilde{f}_0(x_0, x_2^{0},x_3^i) \widetilde{f}_1(x_1,x_2^0,x_3^i) \widetilde{f}_0(x_0,x_2^1,x_3^i) \widetilde{f}_1(x_1,x_2^1,x_3^i)  \phi_{3,j} (x_3^i+p)\, dx_3^i \Big]      \]
 \begin{equation}
     \label{positive3}
 \times g_{(\alpha 2^{m_j})}(x_2^0-x_0-x_1+p)g_{(\alpha 2^{m_j})}(x_2^1-x_0-x_1+p) \, dx_0 dx_1 dx_2^0 dx_2^1   dp.
 \end{equation}
Introducing for the datum $D_2$ the tuple 
$f_{(k,j)}=\widetilde{f}_k$ for $k=0,1$ and $j\in \mathcal{C}$, we may write \eqref{positive3} as
$\Lambda_{D_2,K_1}((f_s)_{s\in S})$, with
$$K_1(u,v,z)=
\sum_{j=1}^J \int_{\R}g_{(\alpha 2^{m_j})}(u+p) g_{(\alpha 2^{m_j})}(v+p)
 {\phi}_{3,j}(z+p){\phi}_{3,j}(p)  \, dp .$$
Proposition \ref{thm:4} implies $\Lambda_{D_2,K_1}((f_s)_{s\in S})\lesssim \alpha^{N}$.

It remains to estimate the term with \eqref{diffw}. We may assume $l_j=k_{j-1}$, because \eqref{diffw} vanishes in the case $k_j=l_j$. With similar 
transformations as for term  \eqref{symmw}, we write the form associated with \eqref{diffw} as
$\Lambda_{D_2,K_2}((f_s)_{s\in S})$ with 
\[K_2(u,v,z)=
\sum_{j=1}^J
\int_{\R}(\nu_j)_{(\alpha)}(u+p,v+p)
 {\phi}_{3,j}(z+p){\phi}_{3,j}(p)  \, dp .\]
 We decompose $\nu_j=\sum_{n\in \Z} \nu_{j,n}$, where
 $$\widehat{\nu_{j,0}}(\xi,\eta)= \widehat{\nu_j}(\xi,\eta) 
 ((\widehat{\chi_{(2^{k_{j-1}})}})^2-(\widehat{\chi_{(2^{k_{j}})}})^2)(\xi+\eta),$$
 and for $n<0$,
 \begin{equation}
     \label{nujn-neq}
     \widehat{\nu_{j,n}}(\xi,\eta)= \widehat{\nu_j}(\xi,\eta) 
 ((\widehat{\chi_{(2^{k_{j-1}+n})}})^2-(\widehat{\chi_{(2^{k_{j-1}+n+1})}})^2)(\xi+\eta)
 \end{equation}
 and for $n>0$,
  \begin{equation}
     \label{nujn-pos}
     \widehat{\nu_{j,n}}(\xi,\eta)= \widehat{\nu_j}(\xi,\eta) 
 ((\widehat{\chi_{(2^{k_{j}+n-1})}})^2-(\widehat{\chi_{(2^{k_{j}+n})}})^2)(\xi+\eta).
 \end{equation}
We split $K_2=\sum_{n\in \Z} K_{2,n}$
accordingly and estimate for each $n$
\[\Lambda_{D_2,K_{2,n}}((f_s)_{s\in S})\lesssim 2^{-|n|}.\]
Upon summing over $n$, we obtain the desired bound for \eqref{diffw}.

We begin with $n=0$.
We have, similarly as in \eqref{takeft1}, for some universal constant $C$,
\[K_2(u,v,z)=C\sum_{j=1}^J \int_{\R^3} \widehat{(\nu_j)_{(\alpha)}}(\xi,\eta)e^{2\pi i (u\xi+v\eta)}\widehat{\phi_{3,j}}(\tau)e^{2\pi i z\tau}\widehat{\phi_{3,j}}(-\tau-\xi-\eta)\, d\xi d\eta d\tau
\]
and thus
\[\widehat{K_{2,0}}(\xi,\eta,\tau)=C\sum_{j=1}^J ((\widehat{\chi_{(\alpha 2^{k_{j-1}})}})^2-(\widehat{\chi_{(\alpha 2^{k_{j}})}})^2)(\xi+\eta)\widehat{(\nu_j)_{(\alpha)}}(\xi,\eta)\widehat{\phi_{3,j}}(\tau)\widehat{\phi_{3,j}}(-\tau-\xi-\eta).
\]
Preparing to apply Proposition \ref{secstickphi}, we note that $K_{2,0}$ is of the form \eqref{ker1space}
with $\rho_j$ defined by
\[{\rho_j}(u_1,u_2,u_3,u_4):=
{(\nu_j)_{(\alpha)}}(u_1,u_2){\phi_{3,j}}(u_3){\phi_{3,j}}(u_4),\]
as can be seen from the Fourier transform side \eqref{ker1}.
We do not attempt to show that $\rho_j$ itself satisfies the assumptions of 
Proposition \ref{secstickphi}, but we split into 
eight
pieces by the distributive law, 
splitting $\nu_j$ into 
two 
pieces as in its definition \eqref{diffw} and each $\phi_{3,j}$ into two
as in its definition \eqref{phi3}. A typical piece is 
\[g_{(\alpha 2^{l_j})}(u_1) g_{(\alpha 2^{k_j})}(u_2)
\chi_{j-1}(u_3)\phi_j(u_4),\]
which satisfies the assumptions of Proposition \ref{secstickphi}, because
\[\int_{\R^2} g_{(\alpha 2^{l_j})}(u_1+p) g_{(\alpha 2^{k_j})}(u_2+p)
\chi_{j-1}(u_3+r)\phi_j(u_4+r)\, dpdr\]
\[\lesssim (g*g)_{(\alpha 2^{1+k_j})}(u_1-u_2) 
2^{-k_j}(1+2^{-k_j}|u_3-u_4|)^{-2}.\]
This along with similar estimates for the other seven pieces completes the bound for $\Lambda_{D_2,K_{2,0}}((f_s)_{s\in S})$ by Proposition \ref{secstickphi}.

We turn to $n>0$. We introduce  artificial factors 
that are constant $1$ where relevant, using that the sequence $k_j$ is well separated, and write
\[\widehat{K_{2,n}}(\xi,\eta,\tau) = C\sum_{j=1}^J ((\widehat{\chi_{(\alpha 2^{k_{j-1}+n+1})}})^2-(\widehat{\chi_{(\alpha 2^{k_{j}+n+1})}})^2)(\xi+\eta)  \]
\[\times 
((\widehat{\chi_{(\alpha 2^{k_{j}+n-1})}})^2-(\widehat{\chi_{(\alpha 2^{k_{j}+n})}})^2)(\xi+\eta)
\widehat{(\nu_j)_{(\alpha)}}(\xi,\eta)\widehat{\phi_{3,j}}(\tau)\widehat{\phi_{3,j}}(-\tau-\xi-\eta) .
\]
This kernel is of the form \eqref{ker1space} with
$$\widehat{\rho_j}(\xi,\eta,\tau, \sigma)=
\widehat{(\nu_{j,n})_{(\alpha)}}(\xi,\eta)\widehat{\phi_{3,j}}(\tau)\widehat{\phi_{3,j}}(\sigma),
$$
with $\nu_{j,n}$  defined in \eqref{nujn-pos}.
We break both functions $\widehat{\phi_{3,j}}$ into pieces as above.
All pieces are done similarly, we discuss a typical piece of $\rho_j$ given by
$$\widehat{\varrho_j}(\xi,\eta,\tau, \sigma)=
\widehat{(\nu_{j,n})_{(\alpha)}}(\xi,\eta)\widehat{\chi_{j-1}}(\tau)\widehat{\phi_{j}}(\sigma) .
$$
Using that $\chi_j$ and $\phi_j$ are even, we have
\[ \int_{\R} {\chi_{j-1}}(u_3+r){\phi_{j}}(u_4+r)\, dr=(\chi_{j-1}*\phi_j)(u_3-u_4)\lesssim 2^{-k_{j}}(1+2^{-k_{j}}|u_3-u_4|)^{-2}.
\]
With Lemma \ref{lemmadecay} below, we obtain
\[\int_{\R^2}|\varrho_j|(u_1+p,u_2+p,u_3+r,u_4+r)\, dpdr\]
\[\lesssim 2^{-n} \alpha^{-1} 2^{- k_{j}}(1+\alpha^{-1} 2^{- k_{j}} |u_1-u_2|)^{-2}
2^{-k_{j}}(1+2^{-k_{j}}|u_3-u_4|)^{-2}.\]
Proposition \ref{secstickphi} gives $\Lambda_{D_2,K_{2,n}}((f_s)_{s\in S})\lesssim 2^{-n}$, as desired.

\begin{lemma}\label{lemmadecay}
We have for every $1\le j\le J$ and every $x,y\in \R$ the estimate 
\[|\nu_{j,n}(x,y)|\lesssim 2^{-n} 2^{- k_j}(1+2^{- k_j}|x-y|)^{-4}2^{-n- k_{j}}(1+2^{-n- k_{j}} |x+y|)^{-4}. \]
\end{lemma}

\begin{proof}[Proof of Lemma~\ref{lemmadecay}]
Scaling by a factor $2^{k_j}$ allows us to assume
$k_j=0$ and  $-1\le m_j\le 0$ and $l_j\le 0$.
We fix $j$ and omit the index $j$.   We thus have to prove
\begin{equation}\label{toprovenun2}
|\nu_{n}(x,y)|\lesssim 2^{-2n} (1+|x-y|)^{-4}(1+2^{-n} |x+y|)^{-4}, 
\end{equation}
where
\begin{equation}\label{defnun2}
\widehat{\nu_{n}}(\xi,\eta)=
((\widehat{\chi_{( 2^{-1})}})^2-(\widehat{\chi})^2)(2^{n}(\xi+\eta))\widehat{\nu}(\xi,\eta)
\end{equation}
with
\begin{equation}\label{defnu2}
\widehat{\nu}(\xi,\eta)= g(2^l\xi) g(\eta)- g(2^{m}\xi)g(2^{m}\eta).
\end{equation}

We claim that for $0\le \alpha\le 4$ and $0\le \beta\le 4$ 
and $|\xi+\eta|\le 1$
 \[|\partial_{(1,-1)}^\alpha \partial_{(1,1)}^\beta \widehat{\nu}(\xi,\eta)|
 \lesssim |\xi+\eta|^{\max\{1-\beta,0\}} (1+|\xi-\eta|)^{-2}.\]
For $\beta>0$, this follows by deriving \eqref{defnu2} 
and using the decay of Gaussians and their derivatives
for those Gaussians whose argument contains $m$ or $k$,
because $-1\le m\le 0$ and $k=0$. Here we also use the fact that whenever $|\xi+\eta|\leq 1$ and $|\xi-\eta|\geq 1$, then the three quantitities $|\xi-\eta|$, $|\xi|$ and $|\eta|$ are comparable.

We next estimate the term with $\beta=0$. By the choice of $m$, the function $\widehat{\nu}$ vanishes on the diagonal $\xi+\eta=0$, and thus the same property holds also for $\partial_{(1,-1)}^\alpha \widehat{\nu}$. Therefore,
\[|\partial_{(1,-1)}^\alpha \widehat{\nu}(\xi,\eta)|
= |\partial_{(1,-1)}^\alpha \widehat{\nu}(\xi,\eta)
-\partial_{(1,-1)}^\alpha \widehat{\nu}((\xi-\eta)/2,-(\xi-\eta)/2)|\]
\[\lesssim |\xi+\eta| \big|\int_{0}^1 \partial_{(1,-1)}^\alpha \partial_{(1,1)} \widehat{\nu}((\xi-\eta)/2+r(\xi+\eta)/2,
-(\xi-\eta)/2+r(\xi+\eta)/2)\,dr\big|\]
\[
\lesssim |\xi+\eta|\sup_{0\le r\le 1}
| \partial_{(1,-1)}^\alpha \partial_{(1,1)} \widehat{\nu}((\xi-\eta)/2+r(\xi+\eta)/2,
-(\xi-\eta)/2+r(\xi+\eta)/2)|
\]
\[\lesssim |\xi+\eta| (1+|\xi-\eta|)^{-2}.\]

Turning to $\widehat{\nu}_n$ as in \eqref{defnun2}, using that
$\widehat{\chi_{( 2^{-1})}}^2-\widehat{\chi}^2$ is supported in $[-2,2]$, we obtain
by differentiating
 \[|\widehat{\nu_{n}}(\xi,\eta)| 
 \lesssim 2^{-n} 1_{|2^n(\xi+\eta)| <1} (1+|\xi-\eta|)^{-2},\]
 \[|\partial_{(1,-1)}^4  \widehat{\nu_n}(\xi,\eta)|
 \lesssim 2^{-n} 1_{|2^n(\xi+\eta)| <1} (1+|\xi-\eta|)^{-2},\]
 \[|\partial_{(1,1)}^4 \widehat{\nu_n}(\xi,\eta)|
 \lesssim 2^{3n} 1_{|2^n(\xi+\eta)| <1} (1+|\xi-\eta|)^{-2},\]
 \[|\partial_{(1,-1)}^4\partial_{(1,1)}^4 \widehat{\nu_n}(\xi,\eta)|
 \lesssim 2^{3n} 1_{|2^n(\xi+\eta)| <1} (1+|\xi-\eta|)^{-2}.\]
Hence, estimating the  Fourier inversion formula crudely by $L^1\to L^\infty$ bounds,
$$|{\nu_n}(x,y)|\lesssim 2^{-2n},$$
$$|x-y|^4 |{\nu_n}(x,y)|\lesssim 2^{-2n},$$
$$|x+y|^4 |{\nu_n}(x,y)|\lesssim 2^{2n},$$
$$|x-y|^4 |x+y|^4 |{\nu_n}(x,y)|\lesssim 2^{2n}.$$
We can summarize these findings into \eqref{toprovenun2},
as can be seen by splitting into four cases
depending on whether $2^n\le |x+y|$ or $2^n> |x+y|$
and depending on whether
$1\le |x-y|$ or $1> |x-y|$. This proves the lemma.
\end{proof}

We finally turn to $n<0$. As in the previous case,
we introduce an artificial factor and write  
\[\widehat{K_{2,n}}(\xi,\eta,\tau) = C\sum_{j=1}^J ((\widehat{\chi_{(\alpha 2^{k_{j-1}+n-1})}})^2-(\widehat{\chi_{(\alpha 2^{k_{j}+n-1})}})^2)(\xi+\eta) \]
\[ \times 
((\widehat{\chi_{(\alpha 2^{k_{j-1}+n})}})^2-(\widehat{\chi_{(\alpha 2^{k_{j-1}+n+1})}})^2)(\xi+\eta)
\widehat{(\nu_j)_{(\alpha)}}(\xi,\eta)\widehat{\phi_{3,j}}(\tau)\widehat{\phi_{3,j}}(-\tau-\xi-\eta).
\]
This kernel is of the form \eqref{ker1space} with
$$\widehat{\rho_j}(\xi,\eta,\tau, \sigma)=
\widehat{(\nu_{j,n})_{(\alpha)}}(\xi,\eta)\widehat{\phi_{3,j}}(\tau)\widehat{\phi_{3,j}}(\sigma)
$$
with $\nu_{j,n}$ as in \eqref{nujn-neq}.
 
We break both functions $\widehat{\phi_{3,j}}$ into pieces as above.
All pieces are done similarly, we discuss a typical piece of $\rho_j$ given by
$$\widehat{\varrho_j}(\xi,\eta,\tau, \sigma)=
\widehat{(\nu_{j,n})_{(\alpha)}}(\xi,\eta)\widehat{\chi_{j-1}}(\tau)\widehat{\phi_{j}}(\sigma) .
$$
With  Lemma \ref{lemmadecay1}, we obtain
\[\int_{\R^2}|\varrho_j|(u_1+p,u_2+p,u_3+r,u_4+r)\, dpdr\]
\[\lesssim 2^{n} \alpha^{-1} 2^{- k_{j}}(1+\alpha^{-1} 2^{- k_{j}} 
|u_1-u_2|)^{-4}2^{-k_{j}}(1+2^{-k_{j}}|u_3-u_4|)^{-2}.\]
Proposition \ref{secstickphi} gives $\Lambda_{D_2,K_{2,n}}((f_s)_{s\in S})\lesssim 2^{n}$, as desired.

\begin{lemma}\label{lemmadecay1}
We have for every $1\le j\le J$ and every $x,y\in \R$ the estimate 
\[|\nu_{j,n}(x,y)|\lesssim  2^n 2^{- k_{j-1}}(1+2^{- k_{j-1}}|x|)^{-4}2^{- k_{j}}(1+2^{- k_{j}} |y|)^{-4} \]
\end{lemma}

\begin{proof}
We split the function 
\begin{equation*}
\widehat{\nu_j}(\xi,\eta)= g(2^{l_j}\xi) g(2^{k_j}\eta)- g(2^{m_j}\xi)g(2^{m_j}\eta)
\end{equation*}
into its two summands and consider the summands separately.
Consider the term
\[ g(2^{l_j}\xi) g(2^{k_j}\eta).\]
Scaling by the factor $2^{l_j}$ in $\xi$ and $2^{k_j}$ in $\eta$  reduces the matter to proving
\begin{equation}\label{E:mun}
|\mu_{n}(x,y)| \lesssim 2^n(1+|x|)^{-4} (1+|y|)^{-4}
\end{equation}
where
\[\wh{\mu_{n}}(\xi,\eta)= (\widehat{\chi}^2-\widehat{\chi_{(2)}}^2)(2^n\xi+2^{n+l}\eta)g(\xi) g(\eta)\]
with $l \leq 0$.
On the support of the function
\[(\widehat{\chi}^2-\widehat{\chi_{(2)}}^2)(2^n\xi+2^{n+l}\eta), \]
we have 
\[|\partial_\xi^\alpha \partial_\eta^\beta 
g(\xi) g(\eta)|\lesssim 2^n (1+|\xi|)^{-4} (1+|\eta|)^{-4}\]
for all $0\le \alpha,\beta\le 4$. By the Leibniz rule, analogous bounds hold for $\wh{\mu_n}$.
The function $\mu_n$ then satisfies the bound~\eqref{E:mun}.
This is the desired estimate for the term
$g(2^{l_j}\xi) g(2^{k_j}\eta)$.

To estimate the term $g(2^{m_j}\xi) g(2^{m_j}\eta)$, we rescale by $2^{m_j}$ in both variables and claim
\[
|\mu_{n}(x,y)| \lesssim 2^n 2^{5l}(1+|x|)^{-4} (1+|y|)^{-4}
\]
where
\[\wh{\mu_{n}}(\xi,\eta)= (\widehat{\chi}^2-\widehat{\chi_{(2)}}^2)(2^{n+l}(\xi+\eta))g(\xi) g(\eta)\]
and $l=k_{j-1}-m_j \leq 0$. This follows similarly as before, using the decay of the Gaussians. As
\[
2^{5l}(1+|x|)^{-4} 
\lesssim 2^{-l}(1+2^{-l}|x|)^{-4},
\]
this completes the proof of the lemma.
\end{proof}

\section{Proof of Proposition~\ref{thm:4} using Propositions~\ref{secstickphi}, \ref{secstickgauss},
and Theorem~1.1 in 
\texorpdfstring{\cite{MR4510172}}{}}
\label{sec:thm:4}

Let $\alpha\ge 1$. Let $J$ be a positive integer and $(k_j)_{j=0}^J$   a   finite increasing sequence of integers with $k_{j-1}+10 \leq k_{j}$
for $1\le j \leq J$, let $(m_j)_{j=1}^J$ be a sequence of real numbers with
$k_j-1\le m_j\le k_j$. For $0\le j\le J$, let $\chi_j$
be a function such that ${(\chi_j)_{(2^{2-k_j})}}$ is a left window
and let $\phi_{j}$ be as in the statement of the proposition, i.e. 
$$ (\widehat{\phi_{j}})^2=(\widehat{\chi_{j-1}})^2-
(\widehat{\chi_{j}})^2.$$
Let a tuple $(f_s)_{s\in S}$ be given as in \eqref{twosides}, \eqref{normal2}
and write $f_{(0,j)}=f_0$, $f_{(1,j)}=f_1$
for any $j \in \mathcal C$.

Taking the Fourier transform, the kernel $K$ of the proposition reads as 
\[ \widehat{K}(\xi,\eta,\tau) =\alpha^{-N}\sum_{j=1}^J \widehat{g_{(\alpha 2^{m_j})}} (\xi)  \widehat{g_{(\alpha 2^{m_j})}}(\eta)  \widehat{\phi_{j}}(\tau) \widehat{\phi_{j}}(-\xi-\eta-\tau).  \]
Define the kernel $K_1$ by
 \[ \widehat{K_1}(\xi,\eta,\tau) :=\alpha^{-N}\sum_{j=1}^J \widehat{g_{(\alpha 2^{m_j})}}(\xi)\widehat{g_{(\alpha 2^{m_j})}}(\eta) (\widehat{\chi_{j-1}}(\tau)\widehat{\chi_{j-1}}(-\tau-\xi-\eta) - \widehat{\chi_{j}}(\tau)\widehat{\chi_{j}}(-\tau-\xi-\eta)  ). \]
Therefore,     on the critical space $\xi+\eta=0$, the kernels are equal, i.e., for all $\xi,\tau$ we have 
\[\widehat{K}(\xi,-\xi,\tau) = \widehat{K_1}(\xi,-\xi,\tau).  \]
By the triangle inequality, it suffices to estimate $\Lambda_{D_2,K-K_1}$ and  $\Lambda_{D_2,K_1}$. 

We begin with the latter. 
Since $\alpha \geq 1$, we observe that it in fact suffices to prove the (stronger) bound $|\Lambda_{D_2,\alpha^{N}K_1}((f_s)_{s\in S})| \lesssim~1$. 
Define the kernel $K_2$ by 
\[\widehat{K_2}(\xi,\eta,\tau) := \sum_{j=1}^J (\widehat{g_{(\alpha 2^{m_{j-1}})}}(\xi)\widehat{g_{(\alpha 2^{m_{j-1}})}}(\eta) - \widehat{g_{(\alpha 2^{m_j})}}(\xi)\widehat{g_{(\alpha 2^{m_j})}}(\eta)) \widehat{\chi_{j-1}}(\tau)\widehat{\chi_{j-1}}(-\tau-\xi-\eta)  \]
and define
\[\widehat{\sigma_{j}}(\xi,\eta,\tau):= \widehat{g_{(\alpha 2^{m_j})}}(\xi)\widehat{g_{(\alpha 2^{m_j})}}(\eta) \widehat{\chi_{j}}(\tau)\widehat{\chi_{j}}(-\tau-\xi-\eta).  \]
Here, we formally set $m_0=k_0$.
By telescoping, we have
\[\alpha^{N}K_1+K_2 = \sigma_0 - \sigma_J.\]
For each $j$, $\Lambda_{D_2,\sigma_{j}}((f_s)_{s\in S})$ equals   
\[  \int_{\R^7}        \Big[\prod_{s\in S} f_s(\Pi_s x)\Big ] g_{(\alpha 2^{m_j})} (x_2^0-x_0-x_1+p) g_{(\alpha 2^{m_j})}(x_2^1-x_0-x_1+p)  \chi_j (x_3^{{0}}+p)\chi_j (x_3^1+p)  \,dxdp, \]
where $x=(x_0,x_1,x_2^0,x_2^1,x_3^0,x_3^1)$. 
This can be {estimated} using a classical Brascamp-Lieb inequality as
\begin{equation}\label{sigmabl}
 |\Lambda_{D_2,\sigma_{j}}((f_s)_{s\in S})| \lesssim \|g_{(\alpha 2^{m_j})}\|^2_1 \|\chi_{j}\|_1^2\prod_{s\in S} \|f_s\|_8\lesssim 1.   
\end{equation}
One can verify this Brascamp-Lieb inequality
by interpolation between estimates that put one of the functions $f_s$ in $L^1$ and all others in $L^\infty$.

The estimate of $\Lambda_{D_2,\alpha^{N}K_1} $ is thus reduced to an estimate of 
$\Lambda_{D_2,K_2}$, which we now proceed to do.
We use the fundamental theorem of calculus to split up a difference of Gaussians with parameters $a,b$ as
\[ {g}(a\xi){g}(a\eta) - {g}(b\xi){g}(b\eta) = 
\int_{a}^{b}   -t\partial_t ( g(t\xi) g(t\eta)) \, \frac{dt}{t}
 = 2\pi \int_{a}^{b} t^2(\xi^2+\eta^2) g(t\xi) g(t\eta)\, \frac{dt}{t}\]
\begin{equation}\label{polid} = 2\pi \int_{a}^{b} t^2(\xi+\eta)^2 g(t\xi) g(t\eta) \,\frac{dt}{t} - 4\pi  \int_{a}^{b} t^2\xi\eta g(t\xi)  g(t\eta)\,
\frac{dt}{t}  .\end{equation}
Using this splitting, in place of $\Lambda_{D_2,K_2}$ we may estimate 
$\Lambda_{D_2,K_3}$ and $\Lambda_{D_2,K_4}$
with
\begin{equation*}
\widehat{K_3}(\xi,\eta,\tau):= \sum_{j=1}^J \int_{\alpha 2^{m_{j-1}}}^{{\alpha 2^{m_{j}}}} t^2(\xi+\eta)^2 g(t\xi) g(t\eta) \, \frac{dt}{t} \;\widehat{\chi_{j-1}}(\tau)\widehat{\chi_{j-1}}(-\tau-\xi-\eta) 
\end{equation*}
and, using $h=g'$ and that $\widehat{h}(\xi)$ is a constant multiple of $\xi\widehat{g}(\xi)$,
\begin{equation*}
\widehat{K_4}(\xi,\eta,\tau):=  \sum_{j=1}^J \int_{\alpha 2^{m_{j-1}}}^{{\alpha 2^{m_{j}}}}
\widehat{h} (t\xi)  \widehat{h}(t\eta)\,
\frac{dt}{t}\;\widehat{\chi_{j-1}}(\tau)\widehat{\chi_{j-1}}(-\tau-\xi-\eta).
\end{equation*}
Proposition \ref{secstickgauss} gives
\[|\Lambda_{D_2,K_3}((f_s)_{s\in S})| \lesssim 1 .\]
 We turn to $\Lambda_{D_2,K_4}$, which we write on the spatial side as
\[
 \sum_{j=1}^J \int_{\alpha 2^{m_{j-1}}}^{{\alpha 2^{m_{j}}}} \int_{\R^5}  \Big[ \int_{\R} 
 \Big[ \prod_{i=0,1}  f_0(x_0, x_2^i,x_3) f_1(x_1,x_2^i,x_3)  \Big ]
 h_{(t)} (x_3+p)\, dx_3 \Big ]^2       \]
 \begin{equation*}
   \times  \chi_{j-1}(x_2^0-x_0-x_1+p)   \chi_{j-1}(x_2^1-x_0-x_1+p) dx_0 dx_1 dx_2^0 dx_2^1 \,   dp \frac{dt}{t}.
 \end{equation*}
Using positivity of the square in this expression, we may dominate
$$|\chi_{j-1}(u)\chi_{j-1}(v)|\le 
\int_{1}^\infty  
g_{(\beta2^{m_{j-1}})}(u)g_{(\beta2^{m_{j-1}})}(v) \beta^{-N}\, \frac {d\beta}\beta.$$
Then   it suffices to estimate for fixed $\beta\ge 1$ the form $\Lambda_{D_2,K_5}$
where 
\begin{equation*}
\widehat{K_5}(\xi,\eta,\tau)  :=   \sum_{j=1}^J \int_{\alpha 2^{m_{j-1}}}^{{\alpha 2^{m_{j}}}} \widehat{h_{(t)}}(\xi)\widehat{h_{(t)}}(\eta) \, \frac{dt}{t}\;\widehat{g_{(\beta 2^{m_{j-1}})}}(\tau)\widehat{g_{(\beta 2^{m_{j-1}})}}(-\tau-\xi-\eta)  .
\end{equation*}
 We introduce a new kernel 
\[\widehat{K_6} (\xi,\eta,\tau)   = \sum_{j=1}^J\widehat{g_{(\alpha 2^{m_{j}})}}(\xi)\widehat{g_{(\alpha 2^{m_{j}})}}(\eta)   \int_{\beta 2^{m_{j-1}}}^{\beta 2^{m_{j}}}  \widehat{h_{(t)}}(\tau){\widehat{h_{(t)}}}(-\tau-\xi-\eta)  \, \frac{dt}{t}.  \]
The kernel $K_6$ is symmetric to $K_5$ under the symmetry \eqref{hiddensymm}.
 We note that, for some $M$, which is even in all variables and symmetric
 under switching the first two variables or switching the second two variables,
\[K_5(x,y,z)=\int_{\R} M(x+p,y+p,z+p,p)\, dp.\]
With $\widetilde{K}_5$ as defined near \eqref{hiddensymm}, we have
\[\widetilde{K}_5(x,y,z)=\int_\R M( \frac {x+y+z}2+p,\frac{x-y+z}2+p,z+p,p)\, dp\]
\[=\int_\R M(-p,-y-p,-\frac{x-z+y}2-p,-\frac{x+z+y}2-p)\, dp,\]
where we obtained the last identity by the  substitution of $p$ by $-p-\frac {x+y+z}2$.
For $\widetilde{K}_5^*$ as defined near \eqref{hiddensymm}, we obtain
\[\widetilde{K}_5^*(x,y,z)=\int_\R M(-p,-z-p,-\frac{x-y+z}2-p,-\frac{x+y+z}2-p)\, dp.\]
Using that $M$ is an even function and that it is invariant under interchanging
the first two entries or the second two entries, we obtain
\[
\widetilde{K}_5^*(x,y,z)=\int_\R M(z+p,p,\frac{x+y+z}2+p,\frac{x-y+z}2+p)\, dp.
\]
Inverting the tilde operation, we identify the kernel
\[{K}_6^*(x,y,z)=\int_{\R} M(z+p,p,x+p,y+p)\, dp.\]
Hence, the star symmetry acts on $M$ by interchanging the first two variables with the second two variables in $M$.

As $\Lambda_{D_2,K_5}((f_s)_{s\in S})$ is positive by the above
construction, it  follows by symmetry that  $\Lambda_{D_2,K_6}$ is positive as well and it suffices to estimate the sum $\Lambda_{D_2,K_5+K_6}$.

We reverse the arguments leading from
$K_2$ to $K_4$, with a Gaussian in place of $\chi_{j-1}$,
and apply these arguments both to $K_5$ and symmetrically to $K_6$.

In place of $\Lambda_{D_2,K_3}$, we obtain the corresponding forms
$\Lambda_{D_2,K_7}$  and $\Lambda_{D_2,K_8}$ with
\begin{equation*}
\widehat{K_7}(\xi,\eta,\tau):= \sum_{j=1}^J \int_{\alpha 2^{m_{j-1}}}^{{\alpha 2^{m_{j}}}} t^2(\xi+\eta)^2 g(t\xi) g(t\eta) \, \frac{dt}{t} \;\widehat{g_{(\beta 2^{m_{j-1}})}}(\tau)\widehat{g_{(\beta 2^{m_{j-1}})}}(-\tau-\xi-\eta),
\end{equation*}
\begin{equation*}
\widehat{K_8}(\xi,\eta,\tau):= \sum_{j=1}^J 
\widehat{g_{(\alpha 2^{m_{j}})}}(\xi)\widehat{g_{(\alpha 2^{m_{j}})}}(\eta)\int_{\beta 2^{m_{j-1}}}^{{\beta 2^{m_{j}}}} 
t^2(\xi+\eta)^2 g(t\tau) g(t(-\tau-\xi-\eta))  \,\frac{dt}{t}. \; 
\end{equation*}
Note that to arrive at $K_8$, in place of symmetry arguments, we 
may also use in place of \eqref{polid} the identity
$$(\xi+\eta)^2+2\tau(\tau+\xi+\eta) =\tau^2+(\tau+\xi+\eta)^2.$$
The forms $\Lambda_{D_2,K_{7}}$ and symmetrically $\Lambda_{D_2,K_{8}}$
are estimated analogously to $\Lambda_{D_2,K_3}$ using Proposition~\ref{secstickgauss}.

Having thus reverted the above steps and having arrived at 
the analogue of $\Lambda_{D_2,K_2}$, we have 
reduced the bound of $\Lambda_{D_2,K_5+K_6}$ to
a bound on $\Lambda_{D_2,K_9}$ with
\[\widehat{K_9}(\xi,\eta,\tau)  \]
\[= \sum_{j=1}^J \left[\widehat{g_{(\alpha 2^{m_{j-1}})}}(\xi)\widehat{g_{(\alpha 2^{m_{j-1}})}}(\eta) - \widehat{g_{(\alpha 2^{m_j})}}(\xi)\widehat{g_{(\alpha 2^{m_j})}}(\eta)\right] \widehat{g_{(\beta 2^{m_{j-1}})}}(\tau)\widehat{g_{(\beta 2^{m_{j-1}})}}(-\tau-\xi-\eta)   \]
\[+ \sum_{j=1}^J \widehat{g_{(\alpha 2^{m_j})}}(\xi)\widehat{g_{(\alpha 2^{m_j})}}(\eta)  \]
\[\times \left[
\widehat{g_{(\beta 2^{m_{j-1}})}}(\tau)\widehat{g_{(\beta 2^{m_{j-1}})}}(-\tau-\xi-\eta)
-\widehat{g_{(\beta 2^{m_{j}})}}(\tau)\widehat{g_{(\beta 2^{m_{j}})}}(-\tau-\xi-\eta)
\right]\]
\begin{equation*}
=\widehat{g_{(\alpha 2^{m_{0}})}}(\xi)\widehat{g_{(\alpha 2^{m_{0}})}}(\eta)\widehat{g_{(\beta 2^{m_{0}})}}(\tau)\widehat{g_{(\beta 2^{m_{0}})}}(-\tau-\xi-\eta) 
\end{equation*}
\begin{equation*}
-\widehat{g_{(\alpha 2^{m_{J}})}}(\xi)\widehat{g_{(\alpha 2^{m_{J}})}}(\eta)\widehat{g_{(\beta 2^{m_{J}})}}(\tau)\widehat{g_{(\beta 2^{m_{J}})}}(-\tau-\xi-\eta) ,
\end{equation*}
where in the last identity we have telescoped the sum. We then obtain
\[|\Lambda_{D_2,K_9}((f_s)_{s\in S})| \lesssim 1\]
by a standard Brascamp-Lieb inequality analogously to the bound \eqref{sigmabl}.  
This completes the bound for  $\Lambda_{D_2,K_1}$.

It remains to estimate $\Lambda_{D_2,K-K_1}$. We have
\[(\widehat{{K}}-\widehat{K_1})(\xi,\eta,\tau) = \alpha^{-N}\sum_{j=1}^J (\widehat{g_{(\alpha 2^{m_j})}} (\xi)  \widehat{g_{(\alpha 2^{m_j})}}(\eta)   \psi_j(\tau, -\tau-\xi-\eta)\]
with
\[\psi_j=\phi_{j}\otimes \phi_{j}  -({\chi_{j-1}}\otimes {\chi_{j-1}} - {\chi_{j}}\otimes {\chi_{j}}  ). \]
Define
\[{\vartheta_{1,j}}  = \phi_{j} - \chi_{j-1}\]
\[{\vartheta_{2,j}} = \chi_{j-1} - (\chi_{j})_{(2^{-4})} \]
 \[\varrho_j  =\psi_j - \vartheta_{1,j}  \otimes \vartheta_{2,j}
- 
\vartheta_{2,j}  \otimes \vartheta_{1,j}  
\]

\[\widehat{{K}_{10}}(\xi,\eta,\tau) = \sum_{j=1}^J (\widehat{g_{(\alpha 2^{m_j})}} (\xi)  \widehat{g_{(\alpha 2^{m_j})}}(\eta)   \wh{\vartheta_{1,j}}(\tau)\wh{\vartheta_{2,j}}(-\tau-\xi-\eta)\]
\[\widehat{{K}_{11}}(\xi,\eta,\tau) = \sum_{j=1}^J (\widehat{g_{(\alpha 2^{m_j})}} (\xi)  \widehat{g_{(\alpha 2^{m_j})}}(\eta)   \wh{\vartheta_{2,j}}(\tau)\wh{\vartheta_{1,j}}(-\tau-\xi-\eta)\]
\[\widehat{{K}_{12}}(\xi,\eta,\tau) = \alpha^{-N}\sum_{j=1}^J (\widehat{g_{(\alpha 2^{m_j})}} (\xi)  \widehat{g_{(\alpha 2^{m_j})}}(\eta)   \wh{\varrho_j}(\tau, -\tau-\xi-\eta)\]
By the triangle inequality, it remains to estimate 
$\Lambda_{D_2,K_{10}}$, $\Lambda_{D_2,K_{11}}$, $\Lambda_{D_2,K_{12}}$
separately.

We begin with $\Lambda_{D_2,K_{10}}$.
Recall that ${(\chi_j)_{(2^{2-k_j})}}$ is a left window.
If $(\tau, -\tau-\xi-\eta)$ is in the support of  $\widehat{\vartheta_{1,j}}\otimes \widehat{\vartheta_{2,j}}$, then 
\[|\tau| \leq 2^{-k_j+2}, \]
\[2^{-k_j+5}\leq |\tau+\xi+\eta| \leq 2^{-k_{j-1}+2}, \]
\[2^{-k_j+4}<|\xi+\eta|<2^{-k_{j-1}+3}.\]
Defining 
 \begin{equation*}
    \vartheta_{3,j}:  = (\phi_{j})_{(2^{-2})}
 \end{equation*}
we have
\[\widehat{K_{10}}(\xi,\eta,\tau) = \sum_{j=1}^J \widehat{g_{(\alpha 2^{m_j})}} (\xi)  \widehat{g_{(\alpha 2^{m_j})}}(\eta)\widehat{\vartheta_{1,j}}(\tau) \widehat{\vartheta_{2,j}}(-\tau-\xi-\eta) \widehat{\vartheta_{3,j}}(\xi+\eta)^2 \]
because the additional factor involving $\widehat{\vartheta_{3,j}}$  is constant 
one on the support of the original summand
in the definition of $K_{10}$. The bound
\[|\Lambda_{D_2,K_{10}}((f_s)_{s\in S})| \lesssim 1\]
then follows from Proposition \ref{secstickphi} applied with
 \[\rho_j:=\widehat{g_{(\alpha 2^{m_j})}} \otimes  \widehat{g_{(\alpha 2^{m_j})}}\otimes\widehat{\vartheta_{1,j}}\otimes \widehat{\vartheta_{2,j}}.\]

The form $\Lambda_{D_2,K_{11}}$ is estimated analoguously
to the form $\Lambda_{D_2,K_{10}}$.
It remains to estimate $\Lambda_{D_2,K_{12}}$. This form is a 
more standard singular Brascamp-Lieb form with a kernel associated with a  H\"ormander-Mikhlin multiplier and we will
apply Theorem 1.1 in \cite{MR4510172}, which was the reason to set $N=2^{18}$.

That theorem will give
\[|\Lambda_{D_2,K_{12}}((f_s)_{s\in S})| \lesssim 1\]
provided
\begin{equation}\label{hoer12}
|\partial^\gamma \widehat{K_{12}}(\xi,\eta,\tau)|
\lesssim |(\xi,\eta,\tau)|^{-|\gamma|} 
\end{equation}
for all multi-indices $\gamma$ of order $0\le |\gamma|\le N$.
The assumption of that theorem that $\Pi_s\Pi^T$ is regular 
for the present datum $D_2$ is satisfied.
It thus remains to show \eqref{hoer12}.

By definition of $\psi_j$ and $\vartheta_{1,j}$, we obtain
\[
\psi_j=\chi_{j-1}\otimes \vartheta_{1,j}
+ \vartheta_{1,j}\otimes \chi_{j-1}+
\vartheta_{1,j}\otimes \vartheta_{1,j}   + {\chi_{j}}\otimes {\chi_{j}}  . \]
Using further the definition of $\varrho_j$ and $\vartheta_{2,j}$, we obtain
\begin{equation}\label{E:rho}
\varrho_j=(\chi_{j})_{(2^{-4})} \otimes \vartheta_{1,j}
+ \vartheta_{1,j}\otimes (\chi_{j})_{(2^{-4})}+
\vartheta_{1,j}\otimes \vartheta_{1,j}   + {\chi_{j}}\otimes {\chi_{j}}.
\end{equation}
Note that $\widehat{\vartheta_{1,j}}$ vanishes outside 
\[[-2^{-k_j+2},2^{-k_j+2}]. \]
Hence $\widehat{\varrho_j}$ is supported on the ball of radius
$2^{10-k_j}$ around the origin. 
In addition, $\wh{\vartheta_{1,j}}$ coincides with $-1$ on $[-2^{-k_j+1},2^{-k_j+1}]$. Using that $(\chi_j)_{(2^{2-k_j})}$ is a left window, we then see that the Fourier transform of the first two terms on the right-hand side of~\eqref{E:rho} is equal to $-1$ on $[-2^{-k_j+1},2^{-k_j+1}]^2$ while the Fourier transform of the last two terms coincides with $1$ on the same set. Therefore,  $\wh{\varrho_j}$ vanishes inside the ball of radius $2^{-k_j}$ around the origin. The support properties of $\wh{\varrho_j}$ together with the estimates $|\wh{\varrho_j}| \lesssim 1$ and $g \lesssim 1$ yield that $|\wh{K_{12}}| \lesssim 1$.

Assume next that $\beta$ is a multi-index with $1\le |\beta|\le N$.
Then $\wh{\rho_j}$ satisfies symbol
estimates adapted to the ball of radius $2^{11-k_j}$ around the origin, namely
\[
|\partial^\beta \wh{\varrho_j}(\tau,\sigma)| \lesssim 2^{k_j|\beta|} 1_{|(\tau,\sigma)| \le 2^{11-k_j}.}
\]
Now assume first $|\xi-\eta|\le |(\xi+\eta,\tau)|.$
Then, using that all derivatives of $g$ up to order $N$
are $\lesssim 1$, and using that $|m_j-k_j|\le 1$ and
$\alpha \geq 1$,
\[|\partial^\beta \widehat{K_{12}}(\xi,\eta,\tau)|\lesssim \alpha^{-N} \sum_{j=1}^J(\alpha 2^{k_j})^{|\beta|} 1_{|(\tau,\tau+\xi+\eta)|\le 2^{11-k_j}}.\]
Using further that $\alpha\ge 1$ and $|\beta| \le N$
we estimate the last display by
\[\lesssim |(\tau,\xi+\eta)|^{-|\beta|}\lesssim |(\tau,\xi,\eta)|^{-|\beta|},\]
where in the last inequality we have used $|\xi-\eta|\le |(\tau, \xi+\eta)|$.
Now assume to the contrary that $|\xi-\eta|\ge |(\xi+\eta,\tau)|.$
Then we use that $|\partial^\beta g(\xi)|\lesssim e^{-|\xi|}$
for all $|\beta|\le N$. Then
\[|\partial^\beta \widehat{K_{12}}(\xi,\eta,\tau)|\lesssim \alpha^{-N} \sum_{j=1}^J(\alpha 2^{k_j})^{|\beta|} e^{-\alpha 2^{k_j}|\xi-\eta|} \lesssim  |\xi-\eta|^{-|\beta| }\sum_{j=1}^J(2^{k_j}|\xi-\eta|)^{|\beta|} e^{ -2^{k_j}|\xi-\eta|}\]
\[\lesssim |\xi-\eta|^{-|\beta|}\sum_{n\in \Z} 2^{n|\beta|}e^{-2^n}
\lesssim |(\xi,\eta,\tau)|^{-|\beta|}.\]

\section{Proof of Propositions  \ref{secstickphi} and \ref{secstickgauss}}
\label{sec:secstick}

The proofs of these propositions have some similarities, so we put them into 
one section and do the second proof analogously to the first.

\subsection{Proof of Proposition  \ref{secstickphi}}\label{sec:9.1}

For $1\le i\le 2$ let  $(a_{i,j})_{j=1}^{J}$ be increasing sequences of positive real numbers, we choose $a_{i,0}>0$ so that  $(a_{i,j})_{j=0}^{J}$
is still increasing.
For $1\leq j \leq J$ let $\rho_{j}:\R^4\to \R$ be a continuous function satisfying \eqref{rhobound}, and pick a further such function 
$\rho_{0}:\R^4\to \R$.
Let $(c_j)_{j=0}^J$ be a well separated increasing sequence of positive numbers. Let $\chi$ be a left window, and let $\phi_{j}$ be a function on $\R$ which for $1\le j \le J$ satisfy $\wh{\phi_j} \geq 0$ and  
 \[(\wh{\phi_{j}})^2 = (\wh{{\chi}_{(c_{j-1})}})^2 
 - (\wh{{\chi}_{(c_j)}})^2. \]
 Let $K$ be defined by \eqref{ker1space} and let a tuple $(f_s)_{s\in S}$ be given as in \eqref{twosides}, \eqref{normal2}.

The integrand of the integral expressing $\Lambda_{D_2,K}((f_s)_{s\in S})$
factors into functions depending on $x_0$
and functions depending on $x_1$. We 
write the integrals in $x_0$ and $x_1$ innermost and separate
these. With $p_0:=p$ and $p_1:=q$ we obtain
\begin{equation}\label{firstlambda29}
\Lambda_{D_2,K}((f_s)_{s\in S})=
\sum_{j=1}^J  \int_{\R^7}  \Big[\prod_{i=0,1} \int_{\R} \Big[ \prod_{q\in \mathcal{C}} f_{(i,q)}(\Pi_{(i,q)}x) \Big] \phi_{j}(x_i+p_i) \,dx_i\Big]    
\end{equation}
\[ \times \rho_{j}(x_2^0+p_0+p_1+r,x_2^1+p_0+p_1+r,x_3^0+r,x_3^1+r)  \, dx_2^0 dx_2^1 dx_3^0 dx_3^1\,dp_0 dp_1 dr  . \] 
Applying the Cauchy-Schwarz inequality in the seven exterior variables
 bounds the last display by the geometric mean of two forms, parameterized
 by $i=0,1$, which with the change of variables $p_{1-i} \to p_{1-i}-p_i-r$ we write as
\[\sum_{j=1}^J \int_{\R^7}  \Big[ \int_{\R} \Big[ \prod_{q\in \mathcal{C}} f_{(i,q)}(\Pi_{(i,q)}x) \Big] \phi_{j}(x_i+p_i)\, dx_i\Big]^2      \]
\begin{equation}
    \label{formaftercs1} \times |\rho_{j}|(x_2^0+p_{1-i},x_{2}^1+p_{1-i},x_3^0+r,x_3^1+r)  \, dx_2^0dx_2^1dx_3^0dx_3^1dp_0dp_1dr .
\end{equation}
Fix $i$ and write $f$ for $f_{(i,j)}$, which thanks to \eqref{twosides}
does not depend on $j$.

Using the decay \eqref{rhobound} for $\rho_{j}$,  we dominate 
\begin{equation}\label{rho23}
\int_{\R^2} |\rho_{j} | (x_2^0+p_{1-i},x_{2}^1+p_{1-i},x_3^0+r,x_3^1+r)\, dp_{1-i} dr
\end{equation}
\[ \lesssim \int_1^\infty \int_1^\infty (g*g)_{(\alpha a_{1,j})} (x_2^0-x_2^1)
(g*g)_{(\beta a_{2,j})}(x_3^0-x_3^1)  \,  \frac {d\alpha}{\alpha^2}  \frac {d\beta}{\beta^2}. \]
It suffices to consider fixed $\alpha$ and $\beta$, and prove uniform bounds
in $\alpha$ and $\beta$  for  \eqref{formaftercs1} with \eqref{rho23} replaced  by
\[(g*g)_{(\alpha a_{1,j})}
(x_2^0-x_2^1)
(g*g)_{(\beta a_{2,j})}(x_3^0-x_3^1) .\] 
Modifying the sequences $a_{i,j}$ if necessary, we may assume
$\alpha=\beta=1$.

Expanding the square in \eqref{formaftercs1} and integrating in $p_i$, our task
becomes to show
\begin{equation}\label{thirddatum}
\sum_{j=1}^J  \int_{\R^6}  \Big[  \prod_{s\in S} f (\Pi_sx)\Big]   (\phi_{j}*\phi_{j})(x_1^0-x_1^1)  
    (g * g)_{a_{1,j}}(x_2^0-x_2^1)  (g * g)_{ a_{2,j}}(x_3^0-x_3^1) \, dx  \lesssim 1,
\end{equation}
where $S$ and $(\Pi_s)_{s\in S}$ are as in the datum $D_{-I}$, which is the datum defined in $\eqref{datumDA}$ in the   case $A=-I$.

 Define the kernels
\begin{equation*}
     K_{1}  := 
        \sum_{j=1}^J ((\chi*\chi)_{(c_{j-1})} - (\chi*\chi)_{(c_{j})})\otimes (g* g)_{( a_{1,j})}\otimes  (g * g)_{( a_{2,j})}, 
\end{equation*} 
\[K_{2}:=\sum_{j=1}^J  (\chi*\chi)_{(c_{j-1})}\otimes ((g* g)_{( a_{1,j-1})} - (g* g)_{( a_{1,j})})\otimes (g * g)_{( a_{2,j})},\]
\[K_{3} := \sum_{j=1}^J  (\chi*\chi)_{(c_{j-1})} \otimes (g* g)_{( a_{1,j-1})} \otimes ((g  * g)_{( a_{2,j-1})} - (g * g)_{( a_{2,j})}),\]
and for $0\leq j \leq J$ also
\[\sigma_{j} = (\chi*\chi)_{(c_{j})}\otimes (g* g)_{( a_{1,j})}\otimes (g  * g)_{( a_{2,j})}.  \]
We have the telescoping identity
\begin{equation}
    \label{telesocpingid}
K_{1}+K_{2}+K_{3} = \sigma_{0} - \sigma_{J}.
\end{equation}

The form  \eqref{thirddatum} to be estimated becomes $\Lambda_{D_{-I}, K_{1}}((f)_{s\in S})$. 
For each $0\leq j \leq J$ one has by a standard Brascamp-Lieb inequality
\[|\Lambda_{D_{-I}, \sigma_{j}}((f)_{s\in S})| \lesssim 1. \]
It then suffices to estimate
the forms associated with $K_2$ and $K_3$ instead. By symmetry, we will only elaborate on  $\Lambda_{D_{-I}, K_{2}}((f)_{s\in S})$.

Next we would like to dominate
  $|(\chi*\chi)_{(c_{j-1})}|$ in these two forms by superposition of  Gaussians in such a way that the cancellation is preserved. To do that,  we will use the identity 
\begin{equation}\label{cor:gaus}
  (g*g)_{(a_{1,j-1})}   - (g*g)_{(a_{1,j})} =
  -\frac 1\pi \int_{a_{1,j-1}}^{a_{1,j}} (h*h)_{(t)}\,
\frac{dt}{t},   
\end{equation}
which follows by taking the Fourier transform of the identity
\[g(a\xi)^2   - g(b\xi)^2
=-\int_a^b  \partial_t g(t\xi)^2 dt=\frac 1\pi \int_a^b (2\pi t\xi  g(t\xi))^2 \frac{dt}t =-\frac 1\pi \int_a^b (\widehat{h}(t\xi))^2 \,\frac{dt}t
\]
for any $a,b>0$.
Using further that $h$ is odd and thus
\[- h*h(x-y)  =  \int_{\R} h(x+p)h(y+p)\,dp , \]
we obtain 
\[\Lambda_{D_{-I}, K_{2}}((f)_{s\in S}) = \frac 1\pi  \sum_{j=1}^J  \int_{a_{1,j-1}}^{a_{1,j}} \int_{\R^5}  \Big[\prod_{i=0,1} \int_{\R} \Big[ \prod_{s(1)=i} f (\Pi_sx) \Big] h_{(t)}(x_2^i+p) dx_2^i\Big]     \]
\begin{equation}
    \label{non-negative}
   \times   (\chi*\chi)_{(c_{j-1})}(x_1^0-x_1^1) (g*g)_{(a_{2,j})}(x_3^0-x_3^1)  \, dx_1^0dx_1^1dx_3^0dx_3^1dp \frac{dt}{t}.
\end{equation}
 
The product over $i=0,1$ has two identical factors and thus is nonnegative.
We may therefore estimate the last display by dominating 
\[|(\chi*\chi)_{(c_{j-1})}|\lesssim \int_1^\infty (g*g)_{(\beta c_{j-1})}\beta^{-5}\, d\beta .\]
It suffices to prove bounds of \eqref{non-negative}
with $(\chi*\chi)_{(c_{j-1})}$ replaced by $(g*g)_{(\beta c_{j-1})}$
uniformly in $\beta$. Fix $\beta$. By changing $c_j$ if necessary,
we may assume $\beta=1$.
Define again kernels
\begin{equation*}
     K_{4}  :=  
        \sum_{j=1}^J ((g * g)_{( c_{j-1})} - (g * g)_{( c_{j})})\otimes (g* g)_{( a_{1,j})}\otimes  (g * g)_{( a_{2,j})}, 
\end{equation*} 
\[K_{5}:=\sum_{j=1}^J  (g * g)_{( c_{j-1})}\otimes ((g* g)_{( a_{1,j-1})} - (g* g)_{( a_{1,j})})\otimes (g * g)_{( a_{2,j})},\]
\[K_{6} := \sum_{j=1}^J  (g * g)_{( c_{j-1})}\otimes (g* g)_{( a_{1,j-1})} \otimes ((g  * g)_{( a_{2,j-1})} - (g * g)_{( a_{2,j})}).\]
Similarly as near \eqref{telesocpingid},
\begin{equation}\label{456}\Lambda_{D_{-I}, K_{4}}((f)_{s\in S})+\Lambda_{D_{-I}, K_{5}}((f)_{s\in S})+\Lambda_{D_{-I}, K_{6}}((f)_{s\in S})
\end{equation}
telescopes into a form that is $\lesssim 1$ by a standard Brascamp-Lieb inequality.
We have seen above that $\Lambda_{D_{-I}, K_{5}}((f)_{s\in S})$ is positive.
By symmetric arguments, the other summands in \eqref{456} are also positive.
Hence each summand is $\lesssim 1$.
This completes the proof of Proposition \ref{secstickphi}.

\subsection{Proof of Proposition  \ref{secstickgauss}}\label{sec:9.2}

 Let a positive integer $J$ be given as well as increasing sequences
 of positive real numbers  $(a_j)_{j=0}^J$, $(b_j)_{j=1}^J$.
Pick $b_0>0$ so that   $(b_j)_{j=0}^J$ is an increasing sequence.
For $1\le j\le J$  let $\phi_{j}$ be given as in \eqref{phistickgauss}.
Let $K$ be defined by \eqref{ker2}.
Let a tuple $(f_s)_{s\in S}$ be given as in \eqref{twosides}, \eqref{normal2}.

 We write
 \[ t^2(\xi+\eta)^2
    g(t\xi)g(t\eta) = t^2(\xi+\eta)^2g(2^{-3/2}t(\xi+\eta))^2g(2^{-1}t(\xi-\eta)) g(2^{-1/2}t\xi)g(2^{-1/2}t\eta)
\]
\[  = -\widehat{h}(2^{-3/2}t(\xi+\eta))^2\widehat{\rho}(2^{-3/2}t(\xi,\eta)).
\]
with
\[
\widehat{\rho}(2^{-3/2}(\xi,\eta)):=\frac{2}{\pi^2}g(2^{-1}(\xi-\eta)) g(2^{-1/2}\xi)g(2^{-1/2}\eta).
\]

Hence passing to the spatial side as near \eqref{ker1}, replacing the arbitrary
sequence $a_j$ by $2^{-3/2}a_j$ to avoid the cumbersome factors $2^{-3/2}$,
\begin{equation*}
{K}(u,v,z)= \sum_{j=1}^J \int_{\R^3}\int_{a_{j-1}}^{a_j}
    h_{(t)}(p)
    h_{(t)}(q) {\rho_{(t)}}(u+p+q+r,v+p+q+r)\, \frac{dt}t
    \phi_{j}(z+r,r) \,dpdqdr. 
\end{equation*}
 We thus have analoguously to \eqref{firstlambda29}
\[ \Lambda_{D_2,K}((f_s)_{s\in S})=
\sum_{j=1}^J  \int_{\R^7}  \int_{a_{j-1}}^{a_j}
\Big[\prod_{i=0,1} \int_{\R} \Big[ \prod_{q\in \mathcal{C}} f_{(i,q)}(\Pi_{(i,q)}x) \Big] h_{(t)}(x_i+p_i) \,dx_i\Big]     \]
\[ \times \rho_{(t)}(x_2^0+p_0+p_1+r,x_2^1+p_0+p_1+r) \phi_j(x_3^0+r,x_3^1+r)  \, \frac{dt}t dx_2^0 dx_2^1 dx_3^0 dx_3^1 dp_0 dp_1 dr  . \] 
Applying the Cauchy-Schwarz inequality 
as in \eqref{formaftercs1}, we need to estimate for $i=0,1$
\[\sum_{j=1}^J \int_{\R^7} \int_{a_{j-1}}^{a_j}
\Big[ \int_{\R} \Big[ \prod_{q\in \mathcal{C}} f_{(i,q)}(\Pi_{(i,q)}x) \Big] h_{(t)}(x_i+p_i)\, dx_i\Big]^2     \]
\[ \times |\rho_{(t)}|(x_2^0+p_{1-i},x_2^1+p_{1-i})|\phi_j|(x_3^0+r,x_3^1+r)\,
\frac{dt}t dx_2^0dx_2^1dx_3^0dx_3^1dp_0dp_1dr .
\]
Thanks to the square, the above integrand is positive and we dominate
\[|\rho_{(t)}|\lesssim g_{(4t)}\otimes g_{(4t)}  \]
and 
\[|\phi_j|\lesssim \int_1^\infty g_{(\beta b_j)}\otimes  g_{(\beta b_j)}\beta^{-3}\, d\beta. \]
It suffices to prove bounds with $g_{(\beta b_j)}\otimes g_{(\beta b_j)}$
in place of $|\phi_j|$ uniformly in $\beta$. Fix $\beta$, we may assume $\beta=1$ by modifying the otherwise arbitrary sequence  $b_j$.

Performing the analogous steps as  leading to \eqref{thirddatum} we end up having to estimate  
$\Lambda_{D_{-I}, K_1}((f)_{s\in S})$,  
where now
\[K_1 := -\sum_{j=1}^J \int_{{a_{j-1}}}^{{a_{j}}} ( h*h)_{(t)}\otimes 
    ( g* g)_{(4 t)}\,\frac{dt}{t}  \otimes (  g* g)_{(b_j)}  . \]
Define 
\[K_2 := -\sum_{j=1}^J \int_{{a_{j-1}}}^{{a_{j}}} ( g*g)_{( t)}\otimes 
    ( h* h)_{(4 t)}\, \frac{dt}{t} \otimes (  g * g)_{(b_j)} ,  \]
    \[K_3 := -\sum_{j=1}^J  (g*g)_{({a_{j-1}})} \otimes  (g *g )_{(4 {a_{j-1}})} \otimes  \int_{b_{j-1}}^{b_j} (h*h)_{(t)} \,  \frac{dt}{t},\]
and for $0\leq j \leq J$ also 
\[\sigma_j = (g*g)_{({a_{j}})} \otimes (g*g)_{(4 {a_{j}})}\otimes (g*g)_{(b_j)} . \]
Then we have the telescoping identity
\begin{equation}
    \label{telescopingid-2}
   K_1 + K_2 +  K_3 = \pi (\sigma_0 - \sigma_J).
\end{equation}
Indeed, this follows with \eqref{cor:gaus}, which gives
\[ K_1 + K_2 =  \pi \sum_{j=1}^J \int_{{a_{j-1}}}^{{a_{j}}} -t\partial_t((g*g)_{( t)}\otimes (g* g)_{(4 t)}    )   \, \frac{dt}{t} \otimes    (g  * g)_{(b_j)}\]
\begin{equation*}
  =    \pi \sum_{j=1}^J  ((g*g)_{ {(a_{j-1})}}\otimes (g* g)_{(4 {a_{j-1}})}  - (g*g)_{ ({a_{j}})}\otimes (g* g)_{(4{a_{j}})} ) \otimes    (g  * g)_{(b_j)}, 
\end{equation*}
and 
\[K_3 = \pi \sum_{j=1}^J (g*g)_{ ({a_{j-1}})} \otimes (g*g)_{(4 {a_{j-1}})}  \otimes ((g*g)_{(b_{j-1})} - (g*g)_{(b_j)}). \]
 
 By the   identity \eqref{telescopingid-2}, 
 \[\Lambda_{D_{-I}, K_1}((f)_{s\in S}) + \Lambda_{D_{-I}, K_2}((f)_{s\in S}) + \Lambda_{D_{-I}, K_3}((f)_{s\in S}) \lesssim 1. \]
 All quantities on the left-hand side are non-negative. 
For $\Lambda_{D_{-I}, K_1}((f)_{s\in S})$, this can be seen as it resulted after an application  of the Cauchy-Schwarz inequality, while for  $\Lambda_{D_{-I}, K_2}((f)_{s\in S})$ and $\Lambda_{D_{-I}, K_3}((f)_{s\in S})$ it follows by symmetry. This gives the desired upper bound
$$\Lambda_{D_{-I}, K_1}((f)_{s\in S}) \lesssim 1. $$

\bibliography{arxivmay25}{}
\bibliographystyle{plain}

\end{document}